\documentclass[12pt,a4paper,draft,twoside,reqno]{amsart}
\usepackage{amssymb}
\usepackage{amsfonts}
\usepackage{amsmath}

\newtheorem{theorem}{Theorem}

\newtheorem{corollary}[theorem]{Corollary}

\newtheorem{definition}[theorem]{Definition}

\newtheorem{proposition}[theorem]{Proposition}
\newtheorem{remark}[theorem]{Remark}

\begin{document}

\title[On structure of linear differential operators]{On structure of linear differential operators, acting in line bundles}
\author[Valentin Lychagin]{Valentin Lychagin}
\address{University of Tromso, Tromso, Norway; Institute of Control Sciences of RAS, Moscow, Russia}
\email{Valentin.Lychagin@matnat.uit.no}
\author[Valeriy Yumaguzhin]{Valeriy Yumaguzhin}
\address{
Program Systems Institute of RAS, Pereslavl'-Zales\-skiy, Russia;  Institute of Control Sciences of RAS, Moscow, Russia}
\email{yuma@diffiety.botik.ru}
\thanks{V. Yumaguzhin is a corresponding author; phone: +79056368327, e-mail: yuma@diffiety.botik.ru }

\maketitle

\begin{abstract}
We study differential invariants of linear differential operators and use
them to find conditions for equivalence of differential operators acting in
line bundles over smooth manifolds with respect to groups of authomorphisms.
\end{abstract}

\tableofcontents

\section{Introduction}

This paper is a continuation of papers \cite{LY2} and \cite{LY3}, where
we analyzed the equivalence of linear differential operators of order $%
k=2, n\geq 3,$ and $k=3, n=2$, acting in line bundles over a smooth
manifolds of dimension $n$. \ In this paper we consider the general case,
when $k\geq 3$ and $n\geq 2.$

Possibly Bernard Riemann \cite{Riem} was the first, who analyzed the
eqiuvalence problem for second order scalar differential operators with
respect to diffeomorphism group.

He found the curvature of the metric, defined by the symbol of differential
operator, as an obstruction to transform differential operators of the
second order to operators with constant coefficients.

In the case, when the dimension of the base manifold equals two,
Pierre-Simon Laplace \cite{Lap} found "Laplace invariants" for the second
order hyperbolic differential operators, which are relative invariants, and
Lev Ovsyannikov in paper \cite{Ovs} found the corresponding invariants.

All invariants for hyperbolic equations in dimension two were found by Nail
Ibragimov in paper \cite{Ibr}.

The method, we used in this paper, is very similar to the method, we used in
papers \cite{LY2},\cite{LY3}.

First of all, we are looking for an affine connection on the base manifold $%
M $ and a linear connection in the line bundle $\xi $, which are naturally
associated with our differential operator.

In the case $k=2,$ we used the Levi-Civita connection on $M,$ given by the
principal symbol of the operator. Then the connection in the line bundle $%
\xi $ was chosen in such a way that the subsymbol of the operator becomes to
be trivial (see \cite{LY2} for more details).

Remark that all these constructions could be applied for operators of the
constant type only. The case of mixed type operators is singular and, as we
know, is singular not only for this method.

In the case $k=3,n=2,$ we've used connections, originally found by Viktor
Wagner \cite{Wag} and Shiing-Shen Chern \cite{Ak},\cite{Ch} in pure
geometrical context, \ instead of the Levi-Civita connection. The linear
connection in the line bundle we found by posed some conditions on the
subsymbol of the operator also.

The regularity conditions to use this method, require also that the operator
has distinct characteristics and once more, the mixed type operators are
excluded from the consideration.

The case of ordinary differential operators, $n=1,$ is also exceptional and
we'll consider it here only\ as an illustration of the general method.

Remark that the case of ordinary differential operators of the second order
was considered by Niky Kamran and Peter Olver in paper \cite{KamOlv} and
the case of linear ordinary differential equations was investigated by
Ernest Wilczynski in book \cite{Wil}.

The second step, when we begin to use these connections, is based on
quantization, defining by the connections \cite{LY2},\cite{LY3}.

The term "quantization" is taken from the wave mechanics \cite{Brog},
where it was used for reconstruction of differential operators from their
symbols.

Necessity of using connections in such type procedures follows directly from
standard requirements \cite{LQ}.

To extend this method for the general case, we, first of all, remark that \
such connections and the corresponding quantizations are exist if our
differential operator has the constant type, i.e. $\rm{GL}$-orbits of
the symbols of the operator do not depend on points of $M.$

In the case of operators of order two it is exactly requirement that
operator is elliptic or (ultra) hyperbolic. For the case $k=3,n=2$ it is the
requirement of distinct characteristics.

In other cases the description of $\rm{GL}$-orbits of the symbols is
the classical algebraic problem of the $\rm{GL}$-classification of $n$%
-ary forms. We use here slightly modified method, suggested in \cite{BL},%
\cite{BL2}, to find rational invariants and regular orbits of such $%
\rm{GL}$-action.

Thus, for constant type operators, the associated connections exist and the
associated quantization allows us to split and represent the differential
operator as a sum of symmetric contravariant tensors.

Then, the more or less routine machinery with using of the connections and
the splitting, allows us to find the field of rational differential
invariants of differential operators.

The third and the last step is based on using of these differential
invariants and \textit{natural coordinates }delivering by differential
invariants.

Namely, the values of $n$-invariants on a differential operator in general
position could be served as local coordinates on $M$, which we call \textit{%
natural}, because coefficients of given operator in these coordinates will
not be changed after applying an automorphism to our operator.

Finally, we introduce the notion of the \textit{natural atlas} and show how
to get the global classification of constant type differential operators and
the corresponding homogeneous differential equations.

\section{Differential operators}

\subsection{Notations}

The notations we use in this paper are similar to notations used in papers %
\cite{LY2},\cite{LY3}.

Let $M$ be an $n$-dimensional manifold and let $\pi :E\left( \pi \right)
\rightarrow M$ be a vector bundle.

We denote by $\tau :TM\rightarrow M$ and $\tau ^{\ast }:T^{\ast
}M\rightarrow M$ \ the tangent and respectively cotangent bundles over
manifold $M,$ and by $\mathbf{1}:\mathbb{R}\times M\rightarrow M\ $we denote
the trivial line bundle $.$

The symmetric and exterior powers of a vector bundle $\pi :E\left( \pi
\right) \rightarrow M$ \ will be denoted by $\mathbf{S}^{k}\left( \pi
\right) $ and $\mathbf{\Lambda }^{k}\left( \pi \right) .$

The module of smooth sections of bundle $\pi $ we denote by $C^{\infty
}\left( \pi \right) ,$ and for the cases tangent, cotangent and the trivial
bundles we'll use the following notations: $\Sigma _{k}\left( M\right)
=C^{\infty }\left( \mathbf{S}^{k}\left( \tau \right) \right) -$ the module
of symmetric $k$-vectors and $\Sigma ^{k}\left( M\right) =C^{\infty }\left( 
\mathbf{S}^{k}\left( \tau ^{\ast }\right) \right) -$ the module of symmetric 
$k$-forms, $\Omega _{k}\left( M\right) =C^{\infty }\left( \mathbf{\Lambda }%
^{k}\left( \tau \right) \right) -$ the module of skew-symmetric $k$-vectors
and $\Omega ^{k}\left( M\right) =C^{\infty }\left( \mathbf{\Lambda }%
^{k}\left( \tau ^{\ast }\right) \right) $ -the module of exterior $k$-forms, 
$C^{\infty }\left( \mathbf{1}\right) =C^{\infty }\left( M\right) .$

\subsection{Jets}

The bundles of $k$-jets of sections of bundle $\pi $ we denote by $\pi _{k}:%
\mathbf{J}^{k}\left( \pi \right) \rightarrow M$ and by $\pi _{k,l}:\mathbf{J}%
^{k}\left( \pi \right) \rightarrow \mathbf{J}^{l}\left( \pi \right) $ we
denote the projection (=reduction) of $k$-jets on $l$-jets, $k\geq l.$

There is the following exact sequence of vector bundles%
\begin{equation*}
\mathbf{0\rightarrow \mathbf{S}^{k}\left( \tau ^{\ast }\right) \otimes }\pi 
\mathbf{\rightarrow J}^{k}\left( \pi \right) \overset{\pi _{k,k-1}}{%
\longrightarrow }\mathbf{J}^{k-1}\left( \pi \right) \rightarrow \mathbf{0,}
\end{equation*}%
connecting bundles of $\left( k-1\right) $ and $k$ -jets.

The inverse limit of the sequence $\mathbf{J}^{k}\left( \pi \right) \overset{%
\pi _{k,k-1}}{\longrightarrow }\mathbf{J}^{k-1}\left( \pi \right) $ \ is called
the bundle of $\infty $-jets: 
\begin{equation*}
\mathbf{J}^{\infty }\left( \pi \right) =\underset{\longleftarrow }{\lim }\,%
\mathbf{J}^{k}\left( \pi \right) .
\end{equation*}

The smooth functions on $\mathbf{J}^{\infty }\left( \pi \right) $ are just
smooth functions on some finite jet bundle. The special vector fields on $%
\mathbf{J}^{\infty }\left( \pi \right) ,$ we call them horizontal, are
extremely important in this paper. \ Namely, let $X$ be a vector field on
manifold $M.$ Then by a \textit{total lift} of $X$ we understand a
derivation $\widehat{X}$ in the algebra of smooth functions $C^{\infty
}\left( \mathbf{J}^{\infty }\left( \pi \right) \right) $ , where 
\begin{equation*}
\widehat{X}:C^{\infty }\left( \mathbf{J}^{k}\left( \pi \right) \right)
\rightarrow C^{\infty }\left( \mathbf{J}^{k+1}\left( \pi \right) \right) ,
\end{equation*}%
for all $k=0,1,..,$ and such the following universal property holds:%
\begin{equation*}
j_{k+1}\left( S\right) ^{\ast }\left( \widehat{X}\left( f\right) \right)
=X\left( j_{k}\left( S\right) ^{\ast }\left( f\right) \right) ,
\end{equation*}%
for any section $S\in C^{\infty }\left( \pi \right)$, and function $f\in
C^{\infty }\left( \mathbf{J}^{k}\left( \pi \right) \right) $, and for $%
k=0,1,...$

Shortly, the universal property could be written in the form:%
\begin{equation*}
j_{k+1}\left( S\right) ^{\ast }\circ \widehat{X}=X\circ j_{k}\left( S\right)
^{\ast }.
\end{equation*}%
Linear combinations of total derivations of the form 
\begin{equation*}
\sum_{i}\lambda _{i}\widehat{X_{i}},
\end{equation*}
where $\lambda _{i}\in C^{\infty }\left( \mathbf{J}^{k}\left( \pi \right)
\right) ,$ we'll call \textit{horizontal vector fields} on the space $%
\mathbf{J}^{k}\left( \pi \right) ,$ and linear combinations of compositions 
\begin{equation*}
\sum_{i_{1}+..+i_{l}\leq N}\lambda _{i_{1}...i_{l}}\widehat{X_{i_{1}}%
}\circ \cdots \circ \widehat{X_{i_{l}}}
\end{equation*}%
we'll call \textit{total differential operators} of order $\leq N$ on the
space $\mathbf{J}^{k}\left( \pi \right) ,$ if all $\lambda
_{i_{1}...i_{l}}\in C^{\infty }\left( \mathbf{J}^{k}\left( \pi \right)
\right) .$

We denote by $\Sigma _{N}\left( \pi \right) $ the module generated by linear
combinations 
\begin{equation*}
\sum_{i_{1}+..+i_{l}=N}\lambda _{i_{1}...i_{N}}\widehat{X_{i_{1}}}%
\cdot \cdots \cdot \widehat{X_{i_{N}}}
\end{equation*}%
of the symmetric products of horizontal vector fields and call them total
symbols of degree $N$ and order $k.$

Let $\omega _{i}$ be exterior (or symmetric) differential $l$-forms on
manifold $M.$ Then linear combinations of pullbacks $\pi _{k}^{\ast }\left(
\omega _{i}\right) $ (which we'll continue to denote by $\omega _{i}$) of
the form $\sum_{i}\lambda _{i}\omega _{i},$ where $\lambda _{i}\in
C^{\infty }\left( \mathbf{J}^{k}\left( \pi \right) \right) ,$ we'll call 
\textit{horizontal differential forms} of degree $l$ and order $k.$ The
modules of such forms we'll denote by $\Omega ^{k}\left( \pi \right) .$

\subsection{Universal constructions}

In the case when the bundle $\pi $ is a tensor bundle or a bundle of
differential operators we'll need the following universal construction which
generalizes the construction of the universal Liouville form on the
cotangent bundle.

Let's consider, for example, the case when $\pi =\Lambda ^{l}\left( \tau
^{\ast }\right) $ is the bundle of exterior differential forms. Then it is
easy to see that there is and unique horizontal $l$-form $\rho _{l}^{a}\in
\Omega ^{l}\left( \pi \right) $ of the order zero and such that%
\begin{equation*}
j_{0}\left( \omega \right) ^{\ast }\left( \rho _{l}^{a}\right) =\omega ,
\end{equation*}%
for all $\omega \in \Omega ^{l}\left( M\right) .$

For the case of contravariant tensors, say $\pi =\tau ,$ the universal
construction goes in the following way.

We define the universal vector field as a horizontal vector field $\nu _{1}$
of order zero such that 
\begin{equation*}
j_{1}\left( X\right) ^{\ast }\left( \nu _{1}\left( f\right) \right) =X\left(
f\right) ,
\end{equation*}%
for all vector fields $X$ on $M$ and all functions $f\in C^{\infty }\left(
M\right) .$

\textbf{Coordinates}

Let $\left( x_{1},..,x_{n}\right) $ be local coordinates on $M$ and $\left(
u^{\alpha }\right) $ or $\left( u_{\alpha }\right) $ be induced coordinates
in the bundle $\Lambda ^{l}\left( \tau ^{\ast }\right) $ or $\Lambda
^{l}\left( \tau \right) .$

Here $\alpha =\left( \alpha _{1}<\alpha _{2}<..<\alpha _{n}\right) $ are
multi indices.

Then $\rho _{1}=\sum u^{i}dx_{i}$ is the universal Liouville form and 
\begin{equation*}
\rho _{l}^{a}=\sum_{\alpha _{1}<\alpha _{2}<..<\alpha _{l}}u^{\alpha
}dx_{\alpha _{1}}\wedge \cdots \wedge dx_{\alpha _{l}}.
\end{equation*}%
For contravariant tensors we have 
\begin{equation*}
\nu _{1}=\sum u_{i}\frac{d}{dx_{i}}
\end{equation*}
and 
\begin{equation*}
\nu _{l}^{a}=\sum_{\alpha _{1}<\alpha _{2}<..<\alpha _{l}}u_{\alpha }\frac{d%
}{dx_{\alpha _{1}}}\wedge \cdots \wedge \frac{d}{dx_{\alpha _{l}}},
\end{equation*}%
where $\frac{d}{dx_{i}}$ are the total derivations.

\subsection{Symbols}

We denote by $\mathbf{Diff}_{k}\left( \xi \right) $ the $C^{\infty }\left(
M\right) $-module of linear differential operators of order $\leq k,\ $%
acting in the vector bundle $\xi $, and by $\pi :Diff_{k}\left( \xi \right)
\rightarrow M$ we denote the bundle of differential operators, thus $%
C^{\infty }\left( \pi \right) =\mathbf{Diff}_{k}\left( \xi \right) $ in this
case.

By \textit{leading or principal symbol} $\sigma _{k}=\rm{smbl}%
_{k}\left( A\right) $ of operator $A\in \mathbf{Diff}_{k}\left( \xi \right) $
we mean the equivalence class 
\begin{equation*}
\rm{smbl}_{k}\left( A\right) =A\rm{mod}\mathbf{Diff}%
_{k-1}\left( \xi \right) .
\end{equation*}

It is known that the symbol could be also viewed as a fibre-wise homogeneous
polynomial of degree $k$ on the cotangent bundle with values in the bundle
of endomorphisms 
\begin{equation*}
\sigma _{k}\in \mathbf{\mathbf{S}^{k}\left( \tau \right) \otimes }\rm{%
End}\left( \xi \right)
\end{equation*}%
and the following sequence

\begin{equation*}
\mathbf{0\rightarrow Diff}_{k-1}\left( \xi \right) \rightarrow \mathbf{Diff}%
_{k}\left( \xi \right) \overset{\rm{smbl}_{k}}{\longrightarrow }%
\mathbf{\Sigma }_{k}\left( M\right) \otimes \rm{End}\left( \xi \right)
\rightarrow \mathbf{0,}
\end{equation*}%
exact.

From now on the case, that we consider in this paper, $\xi $ is a line
bundle. Then the above sequence takes the form 
\begin{equation}
\mathbf{0\rightarrow Diff}_{k-1}\left( \xi \right) \rightarrow \mathbf{Diff}%
_{k}\left( \xi \right) \overset{\sigma }{\rightarrow }\mathbf{\Sigma }%
_{k}\left( M\right) \rightarrow \mathbf{0,}  \label{symbol seq}
\end{equation}%
and because of this we also call \textit{symbols }elements of $\mathbf{%
\Sigma }_{k}=\mathbf{\Sigma }_{k}\left( M\right) $ .

The universal construction, discussed above, in this case gives us a
universal symbols \ $\rho _{k}^{s}$ - horizontal symmetric $k$-vector fields
of order zero such that 
\begin{equation*}
j_{1}\left( \sigma \right) ^{\ast }\left( (df_{1}\cdot \cdots \cdot
df_{k})\rfloor \rho _{k}^{s}\right) =(df_{1}\cdot \cdots \cdot
df_{k})\rfloor \sigma ,
\end{equation*}%
for all symbols $\sigma \in \mathbf{\Sigma }_{k}$ on $M$ and all functions $%
f_{i}\in C^{\infty }\left( M\right) .$

Here we denoted by $\cdot $ the symmetric product and by $\rfloor $ the hook
operator.

In local coordinates the universal symbol $\rho _{k}^{s}$ has the same form
as $\nu _{k}^{a},$ where we changed the exterior product by the symmetric one%
\begin{equation*}
\nu _{k}^{s}=\sum u_{\alpha }\frac{d}{dx_{\alpha _{1}}}\cdot \cdots \cdot 
\frac{d}{dx_{\alpha _{l}}}.
\end{equation*}

In the case, when line bundle is trivial, $\xi =\mathbf{1,}$ and $\pi
=Diff_{k}\left( \mathbf{1}\right) $ is the bundle of scalar differential
operators the universal construction gives us a total differential operator $%
\square _{k}$ on $\mathbf{J}^{0}\left( \pi \right) $ of order $k,$ $\square
_{k}:$ $C^{\infty }\left( \mathbf{J}^{0}\left( \pi \right) \right)
\rightarrow C^{\infty }\left( \mathbf{J}^{k}\left( \pi \right) \right) ,$
such that 
\begin{equation*}
A\left( f\right) =j_{k}\left( A\right) ^{\ast }\left( \square _{k}\left(
f\right) \right) ,
\end{equation*}%
for all $A\in \mathbf{Diff}_{k}\left( \mathbf{1}\right) $ and $f\in
C^{\infty }\left( M\right) .$

Remark, that $\square _{k}:$ $C^{\infty }\left( \mathbf{J}^{l}\left( \pi
\right) \right) \rightarrow C^{\infty }\left( \mathbf{J}^{l+k}\left( \pi
\right) \right) ,$ for all $l=1,2,...,$ and $\square _{k}$ has the form 
\begin{equation*}
\square _{k}=\sum_{\left\vert \alpha \right\vert \leq k}u_{\alpha }\left( 
\frac{d}{dx}\right) ^{\alpha },
\end{equation*}%
where $\left( x_{1},...,x_{n},u_{\alpha },0\leq \left\vert \alpha
\right\vert \leq k\right) $ are canonical coordinates in the bundle $\pi
:Diff_{k}\left( \mathbf{1}\right) \rightarrow M.$

\subsection{Total lifts}

Differential operators $\Delta :C^{\infty }\left( \alpha \right) \rightarrow
C^{\infty }\left( \beta \right) $ of order $k,$ acting from a vector bundle $%
\alpha $ to another vector bundle, say $\beta ,$ could be lifted to
operators $\widehat{\Delta }:C^{\infty }\left( \widehat{\alpha }\right)
\rightarrow C^{\infty }\left( \widehat{\beta }\right) ,$ where $\widehat{%
\alpha }$ and $\widehat{\beta }$ are vector bundles over $\mathbf{J}^{\infty
}\left( \pi \right) ,$ induced by the projection $\pi _{\infty }:\mathbf{J}%
^{\infty }\left( \pi \right) \rightarrow M.$ The operator $\widehat{\Delta }%
, $ which we call \textit{total lift} of $\Delta ,$ defines also by the
universal property: 
\begin{equation*}
j_{k+l}\left( h\right) ^{\ast }\circ \widehat{\Delta }=\Delta \circ
j_{l}\left( h\right) ^{\ast },
\end{equation*}%
for all sections $h\in C^{\infty }\left( \pi \right) $ and $l=0,1,..$

We'll especially use the following two cases. The total lift of the de Rham
operator 
\begin{equation*}
d:\Omega ^{i}\left( M\right) \rightarrow \Omega ^{i+1}\left( M\right)
\end{equation*}
is the \textit{total differential }%
\begin{equation*}
\widehat{d}:\Omega ^{i}\left( \pi \right) \rightarrow \Omega ^{i+1}\left(
\pi \right) .
\end{equation*}

In the another case, when 
\begin{equation*}
d_{\nabla }:C^{\infty }\left( \alpha \right) \rightarrow C^{\infty }\left(
\alpha \right) \otimes \Omega ^{1}\left( M\right)
\end{equation*}
is a covariant differential of a connection $\nabla $ in the bundle $\alpha
, $ operator 
\begin{equation*}
\widehat{d}_{\nabla }:C^{\infty }\left( \widehat{\alpha }\right) \rightarrow
C^{\infty }\left( \widehat{\alpha }\right) \otimes \Omega ^{1}\left( \pi
\right)
\end{equation*}
is the \textit{total covariant differential.}

\subsection{Symbols and Quantization}

Let $\Sigma ^{\cdot }=\oplus _{k\geq 0}\Sigma ^{k}\left( M\right) $ be the
gra\-ded algebra of symmetric differential forms and let $\Sigma ^{\cdot
}\left( \xi \right) =C^{\infty }\left( \xi \right) \otimes \Sigma ^{\cdot }$
be the graded $\Sigma ^{\cdot }$-module of symmetric differential forms with
values in bundle $\xi .$

Assume that we have two connections: connection $\nabla $ on manifold $M$
and linear connection $\nabla ^{\xi }$ in the bundle $\xi .$

Then the covariant differentials 
\begin{equation*}
d_{\nabla }:\Omega ^{1}\left( M\right) \rightarrow \Omega ^{1}\left(
M\right) \otimes \Omega ^{1}\left( M\right) ,
\end{equation*}%
and 
\begin{equation*}
d_{\nabla ^{\xi }}:C^{\infty }\left( \xi \right) \rightarrow C^{\infty
}\left( \xi \right) \otimes \Omega ^{1}\left( M\right)
\end{equation*}%
define two derivations%
\begin{eqnarray*}
d_{\nabla }^{s} &:&\Sigma ^{\cdot }\rightarrow \Sigma ^{\cdot +1}, \\
d_{\nabla ^{\xi }}^{s} &:&\Sigma ^{\cdot }\left( \xi \right) \rightarrow
\Sigma ^{\cdot +1}\left( \xi \right) ,
\end{eqnarray*}%
of degree one in graded algebra $\Sigma ^{\cdot }$ and graded $\Sigma
^{\cdot }$-module $\Sigma ^{\cdot }\left( \xi \right) $ respectively.

Namely, all derivations, as well as these derivations, are defined by their
actions on generators.

We put%
\begin{eqnarray*}
d_{\nabla }^{s} =d &:& C^{\infty }\left( M\right) \rightarrow \Omega
^{1}\left( M\right) =\Sigma ^{1}, \\
d_{\nabla }^{s} &:&\Omega ^{1}\left( M\right) =\Sigma ^{1}\overset{d_{\nabla
}}{\longrightarrow }\Omega ^{1}\left( M\right) \otimes \Omega ^{1}\left(
M\right) \overset{\rm{Sym}}{\longrightarrow }\Sigma ^{2},
\end{eqnarray*}%
and define $d_{\nabla ^{\xi }}^{s}$ as a derivation over $d_{\nabla }^{s}$
such that 
\begin{equation*}
d_{\nabla ^{\xi }}^{s}=d_{\nabla ^{\xi }}:C^{\infty }\left( \xi \right)
\rightarrow C^{\infty }\left( \xi \right) \otimes \Sigma ^{1}.
\end{equation*}

Let's now $\sigma \in \Sigma _{k}$ be a symbol. We define a differential
operator $\widehat{\sigma }\in \mathbf{Diff}_{k}\left( \xi \right) $ as
follows:%
\begin{equation}
\widehat{\sigma }\left( h\right) \overset{\text{def}}{=}\frac{1}{k!}%
\left\langle \sigma ,\left( d_{\nabla ^{\xi }}^{s}\right) ^{k}\left(
h\right) \right\rangle .  \label{Quant}
\end{equation}

Here $h\in C^{\infty }\left( \xi \right) ,\left( d_{\nabla ^{\xi
}}^{s}\right) ^{k}\left( h\right) \in C^{\infty }\left( \xi \right) \otimes
\Sigma ^{k},$and $\left\langle \cdot ,\cdot \right\rangle $ is the natural
pairing 
\begin{equation*}
\Sigma _{k}\otimes C^{\infty }\left( \xi \right) \otimes \Sigma
^{k}\rightarrow C^{\infty }\left( \xi \right) .
\end{equation*}%
Remark that the value of the symbol of the derivation $d_{\nabla ^{\xi
}}^{s} $ on a covector $\theta $ equals to the symmetric product by $\theta $
into the module $\Sigma ^{\cdot }\left( \xi \right) $ and because the symbol
of a composition of operators equals the composition of symbols we get that
the symbol of operator $\widehat{\sigma }$ equals $\sigma .$

We call this operator $\widehat{\sigma }$ a \textit{quantization of symbol }$%
\sigma $ and write $\widehat{\sigma }=\mathcal{Q}\left( \sigma \right) .$

By the construction morphism $\mathcal{Q}\mathit{:}\Sigma _{k}\rightarrow 
\mathbf{Diff}_{k}\left( \xi \right) $ splits sequence (\ref{symbol seq}).

Let's now $A\in \mathbf{Diff}_{k}\left( \xi \right) $ be a differential
operator and $\sigma _{k}\left( A\right) \in \Sigma _{k}$ be its symbol.
Then operator 
\begin{equation*}
A-\mathcal{Q}\left( \sigma _{k}\left( A\right) \right)
\end{equation*}%
has order $\left( k-1\right) ,$ and let $\sigma _{k-1}\left( A\right) \in
\Sigma _{k-1}$ be its symbol.

Then operator $A-\mathcal{Q}\left( \sigma _{k}\left( A\right) \right) -%
\mathcal{Q}\left( \sigma _{k-1}\left( A\right) \right) $ has order $\left(
k-2\right) .$ Repeating this process we get subsymbols $\sigma _{i}\left(
A\right) \in \Sigma _{i},$ $0\leq i\leq k-1,$ such that 
\begin{equation*}
A=\mathcal{Q}\left( \sigma \left( A\right) \right) ,
\end{equation*}%
where 
\begin{equation*}
\sigma \left( A\right) =\oplus _{0\leq i\leq k}\sigma _{i}\left( A\right)
\end{equation*}%
is a \textit{total symbol }of the operator, and $\mathcal{Q}\left( \sigma
\left( A\right) \right) =\sum_{i}\mathcal{Q}\left( \sigma _{i}\left(
A\right) \right) .$

\textbf{Coordinates}

Let $\left( x_{1},...,x_{n}\right) $ be local coordinates in a neighborhood $%
\mathcal{O\subset }M$ and $e\in C^{\infty }\left( \mathcal{O}\right) $ be a
nowhere vanishing section of the line bundle $\xi $ over $\mathcal{O}.$
Denote by $\left( x_{1},...,x_{n},w_{1},..,w_{n}\right) $ induced standard
coordinates in the tangent bundle over $\mathcal{O}.$

Then, $d_{\nabla ^{\xi }}\left( e\right) =e\otimes \theta ,$ where $\theta
=\sum \theta _{i}dx_{i}$ is the connection form, and $d_{\nabla }\left(
dx_{k}\right) =-\sum \Gamma _{ij}^{k}dx_{i}\otimes dx_{j},$ where $\Gamma
_{ij}^{k}$ are the Christoffel symbols of the connection $\nabla =\nabla
^{M}.$

Thus, in coordinates $\left( x,w\right) $ we have $d_{\nabla }^{s}\left(
w_{k}\right) =-\sum \Gamma _{ij}^{k}w_{i}w_{j}$ and the derivations $%
d_{\nabla }^{s}$ and $d_{\nabla ^{\xi }}^{s}$ are of the form:

\begin{eqnarray*}
d_{\nabla }^{s} &=&\sum w_{i}\partial _{x_{i}}-\sum \Gamma
_{ij}^{k}w_{i}w_{j}\ \partial _{w_{k}}, \\
d_{\nabla ^{\xi }}^{s} &=&\sum w_{i}\left( \partial _{x_{i}}+\theta
_{i}\right) -\sum \Gamma _{ij}^{k}w_{i}w_{j}\ \partial _{w_{k}}.
\end{eqnarray*}

\subsection{Group actions}

We consider two groups: $\mathcal{G}\left( M\right) -$ the group of
diffeomorphisms of manifold $M,$ and $\mathbf{Aut}(\xi )-$ groups of
automorphisms of line bundles $\xi $ over $M.$

There is the following sequence of group morphisms 
\begin{equation}
1\rightarrow \mathcal{F}\left( M\right) \rightarrow \mathbf{Aut}(\xi
)\rightarrow \mathcal{G}\left( M\right) \rightarrow 1,
\label{group exact seq}
\end{equation}%
where $\mathcal{F}\left( M\right) \subset C^{\infty }\left( M\right) $ is
the multiplicative group of smooth nowhere vanishing functions on $M$ .

\begin{proposition}
\label{diffeomorphismLift}A diffeomorphism $\psi :M\rightarrow M$ admits a
lifting to an automorphism $\widetilde{\psi }:E\left( \xi \right)
\rightarrow E\left( \xi \right) $ $\ $if and only if $\psi ^{\ast }\left(
w_{1}\left( \xi \right) \right) =w_{1}\left( \xi \right) ,$ where $%
w_{1}\left( \xi \right) \in H^{1}\left( M,\mathbb{Z}_{2}\right) $ is the
first Stiefel-Whitney class of the bundle.
\end{proposition}

\begin{proof}
Remark that a real linear bundle $\xi $ is trivial if and only if the class $%
w_{1}\left( \xi \right) \in H^{1}\left( M,\mathbb{Z}_{2}\right) $ vanishes (%
\cite{MS}). Therefore, in the case when $w_{1}\left( \xi \right) =0$ the
statement of the lemma trivial.

Let now $w_{1}\left( \xi \right) \neq 0$ and let $\psi ^{\ast }\left( \xi
\right) $ be the line bundle induced by a diffeomorphism $\psi .$ Then, we
have 
\begin{equation*}
w_{1}\left( \xi ^{\ast }\otimes \psi ^{\ast }\left( \xi \right) \right)
=w_{1}\left( \xi ^{\ast }\right) +w_{1}\left( \psi ^{\ast }\left( \xi
\right) \right) =w_{1}\left( \xi \right) +\psi ^{\ast }\left( w_{1}\left(
\xi \right) \right) =0,
\end{equation*}%
if \ $\psi ^{\ast }\left( w_{1}\left( \xi \right) \right) =w_{1}\left( \xi
\right) .$

Therefore, any nowhere vanishing section of the bundle $\xi ^{\ast }\otimes
\psi ^{\ast }\left( \xi \right) $ give us an isomorphism between $\ $bundle $%
\xi $ and $\psi ^{\ast }\left( \xi \right) $ covering the identity map and
then the lift of diffeomorphism $\psi .$
\end{proof}

\begin{corollary}
The following sequence of group morphisms 
\begin{equation*}
1\rightarrow \mathcal{F}\left( M\right) \rightarrow \mathbf{Aut}(\xi
)\rightarrow \mathcal{G}_{\xi }\left( M\right) \rightarrow 1,
\end{equation*}

where 
\begin{equation*}
\mathcal{G}_{\xi }\left( M\right) =\left\{ \left. \phi \in \mathcal{G}\left(
M\right) \right\vert \phi ^{\ast }\left( w_{1}\left( \xi \right) \right)
=w_{1}\left( \xi \right) \right\} ,
\end{equation*}%
is exact.
\end{corollary}

\begin{remark}
For the case of complex line bundles this proposition is valid too if we
consider the Chern characteristic classes instead of Stiefel-Whitney classes.
\end{remark}

For the scalar differential operators $A\in \mathbf{Diff}_{k}\left( \mathbf{1%
}\right) $ we consider the standard action of the diffeomorphism group:%
\begin{equation*}
\phi _{\ast }:A\longmapsto \phi _{\ast }\circ A\circ \phi _{\ast }^{-1},
\end{equation*}%
where $\phi _{\ast }=\phi ^{\ast -1}:C^{\infty }\left( M\right) \rightarrow
C^{\infty }\left( M\right) $ is the induced by $\phi $ algebra morphism and $%
\phi ^{\ast }\left( f\right) =f\circ \phi ,$ $f\in C^{\infty }\left(
M\right) .$

For general line bundles and operators $A\in \mathbf{Diff}_{k}\left( \xi
\right) $ the action of the automorphism group is the following.

Let $\widetilde{\phi }$ be an automorphism, $\widetilde{\phi }\in \mathbf{Aut%
}(\xi ),$ covering diffeomorphism $\phi \in \mathcal{G}\left( M\right) .$
Then we define action of $\widetilde{\phi }$ on sections $s\in C^{\infty
}\left( \xi \right) $ as 
\begin{equation*}
\widetilde{\phi }_{\ast }:s\longmapsto \widetilde{\phi }\circ s\circ \phi
^{-1},
\end{equation*}%
and 
\begin{equation*}
\widetilde{\phi }_{\ast }:A\longmapsto \widetilde{\phi }_{\ast }\circ A\circ 
\widetilde{\phi _{\ast }}^{-1},
\end{equation*}%
for differential operators.

\section{Classification of symbols}

In this section we fix a point $a\in M$ on the manifold $M$ and consider
orbits of symbols $\sigma \in S^{k}T_{a}\left( M\right) ,\ $ \ at this point
with respect to general linear group $G=\rm{GL}\left( T_{a}\left(
M\right) \right) ,$ for $k\geq 3$ and $\dim V\geq 2.$

The symbols are homogeneous functions of degree $k$ on the vector space $%
V=T_{a}^{\ast }\left( M\right) ,$ or in other words analytical functions $h$
on the space that satisfy the Euler equations%
\begin{equation}
\delta \left( h\right) =kh,  \label{Euler1}
\end{equation}%
where $\delta $ is the radial vector field on $V.$

\subsection{Euler equations}

Equation (\ref{Euler1}) defines a vector subbundle $\pi _{1}:\mathcal{E}%
_{1}\subset J^{1}\left( V_{0}\right) \rightarrow V_{0},\ $where $V_{0}=\
V\setminus 0,$ in the bundle $\pi _{1}:J^{1}\left( V_{0}\right) \rightarrow
V_{0}$ of $1$-jets of functions on $V_{0}.$

In the canonical coordinates $\left( x_{1},...,x_{n},u,u_{1},..,u_{n}\right) 
$ on $J^{1}\left( V\right) ,$ where $\left( x_{1},...,x_{n}\right) $ are
coordinates on the vector space $V,$ $n=\dim V,$ submanifold $\mathcal{E}%
_{1} $ is given by the equation%
\begin{equation*}
x_{1}u_{1}+\cdots +x_{n}u_{n}-ku=0.
\end{equation*}

Taking the prolongations of the Euler equations we get subbundles $\pi _{i}:%
\mathcal{E}_{i}\subset J^{i}\left( V_{0}\right) \rightarrow V_{0},$ $%
i=1,2,.. $ $k$.

Remark that submanifolds $\mathcal{E}_{l}$ are defined by equations 
\begin{equation*}
\sum_{i}x_{i}u_{\alpha +1_{i}}=\left( k-\left\vert \alpha \right\vert
\right) u_{\alpha },
\end{equation*}%
for all multi indices $\alpha =\left( \alpha _{1},...,\alpha _{n}\right) $
of the length $0\leq \left\vert \alpha \right\vert \leq l-1,$ and where $%
\alpha +1_{i}=\left( \alpha _{1},...\alpha _{i-1},\alpha _{i}+1,\alpha
_{i+1,}...,\alpha _{n}\right) .$

We define now subbundles $\mathcal{E}_{i}\subset J^{i}\left( V_{0}\right) ,$
for $i\geq k+1$ as solutions of the following systems%
\begin{eqnarray*}
\sum_{i}x_{i}u_{\alpha +1_{i}} &=&\left( k-\left\vert \alpha \right\vert
\right) u_{\alpha },\ \left\vert \alpha \right\vert \leq k-1, \\
u_{\beta } &=&0,k+1\leq \left\vert \beta \right\vert \leq i.
\end{eqnarray*}

In this case we also have inclusion of the first prolongations $\mathcal{E}%
_{i}^{\left( 1\right) }\subset \mathcal{E}_{i+1}$ and therefore the system 
\begin{equation*}
x_{1}u_{1}+\cdots +x_{n}u_{n}=ku,u_{\beta }=0,\left\vert \beta \right\vert
=k+1,
\end{equation*}%
defines the formally integrable equation in the sense of Spencer \cite{Sp},\\
\cite{Gold} and analytical (over $V$) solutions of this system are exactly
homogeneous polynomials of degree $k.$

Projections $\pi _{i,i-1}:J^{i}\left( V\right) \rightarrow J^{i-1}\left(
V\right) $ induce projections of prolongations $\pi _{i,i-1}:\mathcal{E}%
_{i}\rightarrow \mathcal{E}_{i-1}$ with kernels (or symbols) $g_{i}\subset
S^{i}\tau ^{\ast }.$

Thus we have exact sequences of vector bundles 
\begin{equation*}
0\rightarrow g_{i}\rightarrow \mathcal{E}_{i}\overset{\pi _{i,i-1}}{%
\longrightarrow }\mathcal{E}_{i-1}\rightarrow 0,
\end{equation*}%
for $i=1,2,...,$ where we put $\mathcal{E}_{0}=J^{0}\left( V_{0}\right) .$

Equation (\ref{Euler1}) shows that 
\begin{equation*}
g_{1}=\left\{ \theta \in T^{\ast }\left( V_{0}\right) ,\delta \rfloor \theta
=0\right\} ,
\end{equation*}%
and 
\begin{equation*}
g_{i}=\left\{ \theta \in S^{i}T^{\ast }\left( V_{0}\right) ,\delta \rfloor
\theta =0\right\} ,
\end{equation*}%
for all $i=2,..k,$ and $g_{i}=0,$ for $i=k+1,...$

Here we denoted by $\delta \rfloor \theta $ the inner product of vector $%
\delta $ and tensor $\theta .$

In other words, if we denote by $\delta ^{0}=\rm{Ann}\left( \delta
\right) \subset T^{\ast }\left( V_{0}\right) \rightarrow V_{0}$ the
subbundle of the cotangent bundle with fibres $\rm{Ann}\left( \delta
_{v}\right) \subset T_{v}^{\ast },$ then 
\begin{equation*}
g_{i}=S^{i}\left( \delta ^{0}\right) ,
\end{equation*}%
for $i=1,..k.$

\subsection{Splitting and invariant frame}

Let $\nabla $ be the standard affine connection on space $V,$ considered as
the affine manifold. Then the above construction allows us to define tensors 
\begin{equation*}
d_{l}f=\frac{1}{l!}\left( d_{\nabla }^{s}\right) ^{l}\left( f\right) \in
\Sigma ^{l}\left( V\right) ,
\end{equation*}%
for any smooth function $f$ on $V.$

The connection $\nabla $ is $G$-invariant and therefore the operators $d_{l}$
are also $G$-invariants:%
\begin{equation*}
A^{\ast }\left( d_{l}f\right) =d_{l}A^{\ast }\left( f\right) ,
\end{equation*}%
for any affine transformation $A:V\rightarrow V,$ and, therefore, for all $%
A\in G.$

Applying these operators to homogeneous functions $H$ of degree $k$ we get
tensors $d_{l}H\in g_{l},$ and because $j_{k}(H)=\left(
H,dH,..,d_{l}H,..,d_{k}H\right) $ we get splitting of the Euler bundles 
\begin{equation*}
\pi _{l}^{\mathcal{E}}:\mathcal{E}_{l}\rightarrow V_{0},
\end{equation*}%
into the direct sum of symbol bundles $\gamma _{i}:g_{i}\rightarrow V_{0},$%
\begin{equation*}
\pi _{l}^{\mathcal{E}}=\mathbf{1\oplus }\gamma _{1}\cdots \mathbf{\oplus }%
\gamma _{l},
\end{equation*}%
$l\leq k.$

Let now $u:J^{0}\left( V\right) =V\times \mathbb{R}\rightarrow \mathbb{R}$
be the standard fibre wise coordinate and let 
\begin{equation*}
\Theta _{l}=\frac{1}{l!}\left( \widehat{d_{\nabla }^{s}}\right) ^{l}\left(
u\right) \in \Sigma ^{l}\left( \mathcal{E}_{l}\right) ,
\end{equation*}%
be horizontal symmetric tensors on the equation.

The, due to the definition of the total lift, we get 
\begin{equation*}
j_{l}\left( H\right) ^{\ast }\left( \Theta _{l}\right) =d_{l}H,
\end{equation*}%
for any homogeneous polynomial $H.$

In other words, tensors $\Theta _{l}$ are universal differentials of $l$-th
order, $l=0,1,...$

In the standard jet coordinates $\left( x,u,...,u_{a},..\right) $ on the jet
spaces the universal tensors $\Theta _{l}$ are of the following form%
\begin{equation*}
\Theta _{l}=\sum_{\left\vert \alpha \right\vert =l}u_{\alpha }\frac{%
dx^{\alpha }}{\alpha !},
\end{equation*}%
and it is easy to check that 
\begin{equation}
\widehat{\delta }\rfloor \Theta _{l}=\left( k-l+1\right) \Theta _{l-1},
\label{Qind}
\end{equation}%
for all $l=1,...,k.$

In particular, $\widehat{\delta }\rfloor \Theta _{1}=k\Theta _{0},$ and in
the domain, where $\Theta _{0}=u\neq 0,$ we can represent any horizontal
vector field $X$ as sum: 
\begin{equation}
X=X_{0}+\frac{\Theta _{1}\left( X\right) }{k\Theta _{0}}\widehat{\delta },
\label{Xdecom}
\end{equation}%
where $X_{0}\in \ker \Theta _{1}$ is also horizontal field.

Therefore, due to this splitting, any horizontal $1$-form $\lambda $ could
be decomposed into the sum%
\begin{equation}
\lambda =\lambda ^{0}+\frac{\lambda \left( \widehat{\delta }\right) }{%
k\Theta _{0}}\Theta _{1},  \label{split}
\end{equation}%
where the form $\lambda ^{0}$ is considered as a form $\lambda $ restricted
on $\ker \Theta _{1}.$

Applying formula (\ref{Qind}), we get 
\begin{eqnarray*}
\Theta _{2} &=&\Theta _{2}^{0}+\frac{k-1}{2k\Theta _{0}}\Theta _{1}^{2}, \\
\Theta _{3} &=&\Theta _{3}^{0}+\frac{1}{k\Theta _{0}}\Theta _{2}^{0}\cdot
\Theta _{1}+\frac{k-1}{3!\left( k\Theta _{0}\right) ^{2}}\Theta _{1}^{3},
\end{eqnarray*}%
where $\Theta _{2}^{0}$ and $\Theta _{3}^{0}$ are quadratic and cubic
differential forms on horizontal vector fields from $\ker \Theta _{1}.$

These two tensors $\Theta _{2}^{0}$ and $\Theta _{3}^{0}$ will be important
for us. \ We say that a point $a_{2}\in \mathcal{E}_{2}$ is \textit{regular}
or \textit{singular }if the quadratic form $\Theta _{2}$ at this point is
regular and $u\left( a_{2}\right) \neq 0,$ and singular in the opposite case.

Remark, that regularity $\Theta _{2}$ is equivalent to regularity of $\Theta
_{2}^{0}.$ Indeed, assume that $X\in \ker \Theta _{2}$ and has decomposition
(\ref{Xdecom}). Then%
\begin{equation*}
X\rfloor \Theta _{2}=X_{0}\rfloor \Theta _{2}^{0}+\frac{\left( k-1\right)
\Theta _{1}\left( X\right) }{k\Theta _{0}}\Theta _{1},
\end{equation*}%
and therefore $X\in \ker \Theta _{2}$ if and only if $\Theta _{1}\left(
X\right) =0,$ i.e. $X=X_{0},$ and $X_{0}\in \ker \Theta _{2}^{0}.$

Denote by $\mathcal{E}_{2}^{0}\subset \mathcal{E}_{2}$ the domain of regular
points and by $\widehat{\Theta _{2}^{0}}$ the inverse tensor to $\Theta
_{2}^{0}$.

Then 
\begin{equation*}
\lambda =\widehat{\Theta _{2}^{0}}\rfloor \Theta _{3}^{0},
\end{equation*}%
is a horizontal differential $1$-form (on $\ker \Theta _{1}$) over the
regular domain $\mathcal{E}_{2}^{0}$.

Let $\widehat{\lambda _{1}}$ be the horizontal vector field in $\ker \Theta
_{1}$ dual to $\lambda ,$ i.e. $\lambda =\widehat{\lambda }\rfloor \Theta
_{2}^{0},$ or $\widehat{\lambda }=\lambda \rfloor \widehat{\Theta _{2}^{0}},$
and let 
\begin{equation*}
\Upsilon =\widehat{\lambda }\rfloor \Theta _{3}^{0},
\end{equation*}%
be a horizontal quadratic form on $\ker \Theta _{1}$.

Denote by $\Phi $ operator that corresponds to this form i.e. a linear
operator acting on horizontal vector fields in $\ker \Theta _{1}$ and such
that 
\begin{equation*}
\Theta _{3}^{0}\left( \widehat{\lambda },X,Y\right) =\Theta _{2}^{0}\left(
\Phi X,Y\right) ,
\end{equation*}%
for all horizontal vector fields $X,Y$ in $\ker \Theta _{1}.$

This operator as well as all above constructions well defined over the
regular domain only.

We say that a point $a_{3}\in \mathcal{E}_{3}$ is \textit{regular }if its
projection $a_{2}=\pi _{3,2}^{\mathcal{E}}\left( a_{3}\right) $ on the space
of 2-jets belongs to the regular domain $\mathcal{E}_{2}^{0}$ and horizontal
vector fields 
\begin{equation*}
e_{1}=\widehat{\lambda },e_{2}=\Phi \left( e_{1}\right) ,...,e_{n-1}=\Phi
\left( e_{n-2}\right) ,
\end{equation*}%
where $n=\dim V,$ are linear independent.

\subsection{Invariants of homogeneous forms}

Let $\mathcal{E}_{3}^{0}\subset \mathcal{E}_{3}$ be the domain of regular
3-jets. This domain is non empty and defined by some number of algebraic
inequalities. Therefore, $\mathcal{E}_{3}^{0}$ is dense into $\mathcal{E}%
_{3}.$
\newpage
\begin{theorem} $\phantom{a}$
\begin{enumerate}
\item Horizontal vector fields $e_{1},...,e_{n-1}$ and $e_{n}=\widehat{%
\delta }$ constitute $G$-invariant frame over regular domain $\mathcal{E}%
_{3}^{0}.$

\item Function $\Theta _{0}=u$ and coefficients of tensors $\Theta
_{l}^{0},l=2,...,k$ in the frame $\left( e_{1},...,e_{n-1}\right) $ in $\ker
\Theta _{1}$ are $G$-invariants. They are rational functions over the
regular domain $\mathcal{E}_{3}^{0}$ and they generate all rational
differential $G$-invariants of homogeneous forms on $V.$
\end{enumerate}
\end{theorem}

The total lift of the symmetric covariant differentials 
\begin{equation*}
d_{\nabla }^{s}:\Sigma ^{l}(V)\overset{d_{\nabla }}{\rightarrow }\Sigma
^{1}\left( V\right) \otimes \Sigma ^{l}(V)\overset{\text{Sym}}{\longrightarrow }%
\Sigma ^{l+1}(V)
\end{equation*}%
allow us to reconstruct universal tensor $\Theta _{i+1}$ under condition
that we know tensor $\Theta _{i}:\widehat{d_{\nabla }^{s}}\left( \Theta
_{i}\right) =\Theta _{i+1}.$

We'll extend the notion of regularity. Namely, we say that a 4-jet $a_{4}\in 
\mathcal{E}_{4}$ is \textit{regular }if its projection on the space of
3-jets is regular, $a_{3}\in \mathcal{E}_{3}^{0},$ and there are invariants $%
J_{1},...,J_{n}=\Theta _{0}$ of order $\leq 3,$ such that $\widehat{d}%
J_{1}\wedge \cdots \wedge \widehat{d}J_{n}\neq 0$ at point $a_{4}.$ Denote
by $\mathcal{E}_{4}^{0}\subset \mathcal{E}_{4}$ the domain of regular
4-jets. It is Zarissky open and therefore dense domain in $\mathcal{E}_{4}.$

These\ invariants $J_{1},J_{2},...,J_{n}=\Theta _{0}$ are in \textit{general
position}, i.e.%
\begin{equation*}
\widehat{d}J_{1}\wedge \cdots \wedge \widehat{d}J_{n}\neq 0
\end{equation*}%
in an open and dense domain in $\mathcal{E}_{4}.$

Then functions 
\begin{equation*}
J_{ab}=e_{a}\left( J_{b}\right) ,
\end{equation*}%
where $a,b=1,...,n,$ are rational $G$-invariants and they defined the
invariant frame%
\begin{equation*}
e_{a}=\sum_{b}J_{ab}\frac{d}{dJ_{b}},
\end{equation*}%
where $\frac{d}{dJ_{b}}$ are the Tresse derivatives.

Taking the total covariant derivatives 
\begin{equation*}
\widehat{\nabla }_{e_{a}}\left( e_{b}\right) =e_{b}\rfloor \widehat{%
d_{\nabla }}\left( e_{b}\right) ,
\end{equation*}%
and decomposing them in the invariant frame we get%
\begin{equation*}
\widehat{\nabla }_{e_{a}}\left( e_{b}\right) =\sum_{c}\Gamma _{ab}^{c}e_{c},
\end{equation*}%
where Christoffel symbols $\Gamma _{ab}^{c}$ are rational differential $G$%
-invariants too.

This data 
\begin{equation}
J=\left( J_{1},J_{2},...,J_{n}=\Theta _{0}\right) ,U=\left( J_{ab}\right)
,\Gamma =\left( \Gamma _{ab}^{c}\right)  \label{data}
\end{equation}

completely defines invariant frame by $U,$ the total covariant differential $%
\widehat{d_{\nabla }}$ \ by $\Gamma $ and, therefore, all universal tensors $%
\Theta _{l},$ because $J_{n}=\Theta _{0}.$

\begin{theorem}
For a given data (\ref{data}) all rational differential invariants of
rational functions of degree$\ k$ $\geq 3,n\geq 2,$ are rational functions
of invariants $J$ and Tresse derivatives%
\begin{equation*}
\frac{dU}{dJ},\quad\frac{d\Gamma }{dJ}.
\end{equation*}
\end{theorem}

Remark also that 
\begin{equation*}
\dim \mathcal{E}_{3}=\frac{n\left( n^{2}+3n+8\right) }{6},
\end{equation*}%
and therefore codimensions of regular orbits are 
\begin{equation*}
\dim \mathcal{E}_{3}-n^{2}\geq n,
\end{equation*}%
for $n\geq 2.$

\subsection{Orbits of homogeneous forms}

Let $h$ be a homogenous function of order $k$. We say that $h$ is a \textit{%
regular function} if its 4-jet $j_{4}\left( h\right) $ has non empty
intersection with regular domain $\mathcal{E}_{4}^{0}.$

Let now $I$ be an invariant of order $l,$ then by $I\left( h\right)
=j_{l}\left( h\right) ^{\ast }\left( I\right) $ we will denote the value of
this invariant on the homogeneous function $h.$

Remark that $I\left( h\right) $ is a rational function defined on open and
dense set in $V.$

Respectively, by $J\left( h\right) ,U\left( h\right) ,\Gamma \left( h\right) 
$ we'll denote the values of data (\ref{data}) on function $h.$

Invariants $J$ are in general position, therefore there is an open domain $%
\mathcal{O}_{h}\mathcal{\subset }V,$ where functions $J\left( h\right)
=\left( J_{1}\left( h\right) ,...,h\right) $ are coordinates and, therefore,
functions $U\left( h\right) $ and $\Gamma \left( h\right) $ are functions of 
$J:$ $U=U\left( J\right) ,\Gamma =\Gamma \left( J\right) .$

Geometrically, we'll consider the following rational map 
\begin{eqnarray*}
D_{h} &:&V\rightarrow \mathbb{R}^{N}, \\
D_{h} &:&v\in V\mapsto \left( R=J\left( h\right) \left( v\right) ,Q=U\left(
h\right) \left( v\right) ,S=\Gamma \left( h\right) \left( v\right) \right) ,
\end{eqnarray*}%
where 
\begin{equation*}
N=\frac{n\left( n^{2}+2n+3\right) }{2},
\end{equation*}%
and the standard coordinates in $\mathbb{R}^{N}$ we denoted by $\left(
R,Q,S\right) ,$ and $R=\left( R_{1},...,R_{n}\right) ,$ $Q=\left(
Q_{ab}\right) ,S=\left( S_{ab}^{c}\right) .$

Then the image $\Pi _{h}\subset \mathbb{R}^{N}$ of $D_{h}$ is an algebraic
variety and the image $D_{h}\left( \mathcal{O}_{h}\right) $ of the domain $%
\mathcal{O}_{h}$ is the graph of functions $U\left( R\right) ,\Gamma \left(
R\right) :$%
\begin{equation*}
Q=U\left( R\right) ,S=\Gamma \left( R\right) .
\end{equation*}

Remark that linear transformations $A\in G$ leave invariant algebraic
relations between differential invariants and therefore do not change the
algebraic manifold $\Pi _{h}:$%
\begin{equation*}
\Pi _{A^{\ast }\left( h\right) }=\Pi _{h}.
\end{equation*}

\begin{theorem}
Two regular polynomials $h$ and $h^{\prime }$ of degree $k\geq 3,n\geq 2,$
are $G$-equivalent if and only if $\ \Pi _{h}=\Pi _{h^{\prime }},$ for the
some invariants $J_{1},...,J_{n}=\Theta _{0}$ in general position.
\end{theorem}

\begin{proof}
The above remark shows the necessity of the theorem condition.

Let's now $\Pi _{h}=\Pi _{h^{\prime }}$. Then domains $\mathcal{O}_{h}$ and $%
\mathcal{O}_{h^{\prime }}$ we choose in such a way that $\Pi \left( \mathcal{%
O}_{h}\right) =\Pi \left( \mathcal{O}_{h^{\prime }}\right) $ is common for
both functions. Over domains $\mathcal{O}_{h}$ we have%
\begin{equation*}
U_{ab}=U_{ab}\left( z_{1},...,z_{n}\right) ,\Gamma _{ab}^{c}=\Gamma
_{ab}^{c}\left( z_{1},...,z_{n}\right)
\end{equation*}%
in coordinates 
\begin{equation*}
z_{1}=J_{1}\left( h\right) ,...,z_{n-1}=J_{n-1}\left( h\right) ,,,,z_{n}=h,
\end{equation*}%
and \ over domain $\mathcal{O}_{h^{\prime }}$ we get the same relations and
functions 
\begin{equation*}
U_{ab}=U_{ab}\left( z_{1}^{\prime },...,z_{n}^{\prime }\right) ,\Gamma
_{ab}^{c}=\Gamma _{ab}^{c}\left( z_{1}^{\prime },...,z_{n}^{\prime }\right)
\end{equation*}%
in coordinates%
\begin{equation*}
z_{1}^{\prime }=J_{1}\left( h^{\prime }\right) ,...,z_{n-1}^{\prime
}=J_{n-1}\left( h^{\prime }\right) ,,,,z_{n}^{\prime }=h^{\prime }.
\end{equation*}%
Therefore, the diffeomorphism $A:\mathcal{O}_{h}\rightarrow \mathcal{O}%
_{h^{\prime }},$ where $z\rightarrow z^{\prime },$ transforms $h^{\prime }$
to $h$ and preserves the affine connection. Thus, $A\in G$ and $A^{\ast
}\left( h^{\prime }\right) =h.$
\end{proof}

\begin{proposition}
Isotropy Lie algebra of an $\rm{GL}\left( V\right) $-orbit of a regular
homogeneous $k$-form is trivial for $k\geq 3$ and $n\geq 2.$
\end{proposition}

\begin{proof}
Any vector field $X\in \mathfrak{gl}\left( V\right) $ from the isotropy Lie
algebra of a regular homogeneous form of degree $k\geq 3,$ preserves also
invariants determining orbits of 3-jets . Then, $X\left( J_{i}\left(
H\right) \right) =0$ in an open domain and, therefore, $X=0.$
\end{proof}

\begin{corollary}
Let $H$ be a regular homogenous $k$-form, $k\geq 3$. Then the isotropy group
of $H$ in $\rm{GL}\left( V\right) $ is finite.
\end{corollary}

\begin{proof}
The isotropy group is discrete, Zarissky closed and therefore finite.
\end{proof}

\begin{corollary}
\label{ISOTROPY}Let $H$ be a regular homogenous $k$-form, $k\geq 3.$ Then
there is a neighborhood $\mathcal{O}_{H}$ of $H$ in the space of homogeneous
forms of degree $k$ such that for any form $H^{\prime }\in \mathcal{O}_{H},$
which belongs to the $\rm{GL}\left( V\right) $-orbit of $H,$ there is
and unique linear transformation $A\in \rm{GL}\left( V\right) ,$ such
that $A^{\ast }\left( H^{\prime }\right) =H.$
\end{corollary}

\subsection{Example: Binary forms}

To make the general case more transparent we'll illustrate this approach on
example of binary forms. The results in this section are very closed to
classification given in \cite{BL2}.

The Euler equation for binary $k$-forms $\mathcal{E}_{k}\subset J^{k}\left(
V_{0}\right) $ has dimension $k+3$ with coordinates : $%
x_{1},x_{2},u,u_{1,0},...,u_{k,0}.$

The universal tensors $\Theta _{l}$ are of the form%
\begin{equation*}
\Theta _{l}=\sum_{\alpha _{1}+\alpha _{2}=l}u_{\alpha _{1},\alpha _{2}}\frac{%
dx_{1}^{\alpha _{1}}}{\alpha _{1}!}\frac{dx_{2}^{\alpha _{2}}}{\alpha _{2}!},
\end{equation*}%
for general $l,$ and especially 
\begin{eqnarray*}
\Theta _{0} &=&u_{0,0},\ \Theta _{1}=u_{1,0}dx_{1}+u_{0,1}dx_{2}, \\
\Theta _{2} &=&\frac{1}{2}\left( u_{2,0}dx_{1}^{2}+2u_{1,1}dx_{1}\cdot
dx_{2}+u_{0,2}dx_{2}^{2}\right) , \\
\Theta _{3} &=&\frac{1}{6}\left( u_{3,0}dx_{1}^{3}+3u_{2,1}dx_{1}^{2}\cdot
dx_{2}+3u_{1,2}dx_{1}\cdot dx_{2}^{2}+u_{0,3}dx_{2}^{3}\right) ,
\end{eqnarray*}%
for small $l.$

To construct the $G$- invariant frame we'll take two horizontal vector
fields 
\begin{equation*}
\widehat{\delta }=x_{1}\frac{d}{dx_{1}}+x_{2}\frac{d}{dx_{2}},
\end{equation*}%
and 
\begin{equation*}
\eta =u_{0,1}\frac{d}{dx_{1}}-u_{1,0}\frac{d}{dx_{2}},
\end{equation*}%
in order to have horizontal form $\Theta _{1}$ as an element of the dual
basis.

In splitting (\ref{split}) we have $\theta ^{0}\left( \delta \right) =0,$
therefore the second form in the $G$- invariant coframe should be
proportional to form%
\begin{equation*}
\rho =x_{2}dx_{1}-x_{1}dx_{2}.
\end{equation*}%
Then for the basic forms $dx_{1}$ and $dx_{2},$ we get%
\begin{eqnarray*}
dx_{1} &=&\frac{u_{0,1}}{ku}\rho +\frac{x_{1}}{ku}\Theta _{1}, \\
dx_{2} &=&-\frac{u_{1,0}}{ku}\rho +\frac{x_{1}}{ku}\Theta _{1}
\end{eqnarray*}%
and therefore the quadratic form $\Theta _{2}^{0}$ will be the following%
\begin{equation*}
\Theta _{2}^{0}=\frac{K_{2}}{\left( ku\right) ^{2}}\rho ^{2},
\end{equation*}%
where 
\begin{equation*}
K_{2}=\frac{1}{2}\left(
u_{2,0}u_{0,1}^{2}-2u_{1,1}u_{1,0}u_{0,1}+u_{0,2}u_{1,0}^{2}\right) .
\end{equation*}%
In the similar way we get%
\begin{equation*}
\Theta _{3}^{0}=\frac{K_{3}}{\left( ku\right) ^{3}}\rho ^{3},
\end{equation*}%
where 
\begin{equation*}
K_{3}=\frac{1}{6}\left(
u_{3,0}u_{0,1}^{3}-3u_{2,1}u_{0,1}^{2}u_{1,0}+3u_{1,2}u_{0,1}u_{1,0}^{2}-u_{0,3}u_{1,0}^{3}\right) .
\end{equation*}%
Therefore, the invariant form $\lambda $ equals 
\begin{equation*}
\lambda =\frac{1}{3}\frac{K_{3}}{K_{2}}\frac{\rho }{ku},
\end{equation*}%
or substituting $\lambda $ instead of $\rho $ we get%
\begin{equation*}
\Theta _{2}^{0}=\frac{9K_{2}^{3}}{K_{3}^{2}}\rho ^{2},\ \Theta _{3}^{0}=%
\frac{27K_{2}^{3}}{K_{3}^{2}}\rho ^{3}.\ 
\end{equation*}%
These forms are $G$-invariants, therefore:
\begin{itemize}
\item functions 
\begin{equation*}
J_{0}=u,\ J_{3}=\frac{K_{3}^{2}}{K_{2}^{3}}
\end{equation*}%
are differential invariants of binary forms,
\item horizontal forms 
\begin{equation*}
\left\langle \lambda ,\Theta _{1}\right\rangle
\end{equation*}%
give us $G$-invariant coframe, and
\item total vector fields%
\begin{equation*}
\left\langle \widehat{\lambda },\widehat{\delta }\right\rangle ,
\end{equation*}%
where%
\begin{equation*}
\widehat{\lambda }=\frac{3K_{3}}{2K_{2}^{2}}\left( u_{0,1}\frac{d}{dx_{1}}%
-u_{1,0}\frac{d}{dx_{2}}\right) ,
\end{equation*}%
form $G$-invariant frame.
\end{itemize}

Applying $\widehat{\lambda },\widehat{\delta }$ to differential invariants
we get also invariants. Thus,%
\begin{eqnarray*}
\widehat{\delta }\left( J_{0}\right) &=&kJ_{0},\ \widehat{\delta }\left(
J_{3}\right) =-kJ_{3}, \\
\widehat{\lambda }\left( J_{0}\right) &=&0,\ \widehat{\lambda }\left(
J_{3}\right) =J_{4},
\end{eqnarray*}%
where $J_{4}$ is a new invariant of order $4.$

Also, writing down the tensors $\Theta _{l}^{0}$ in invariant coframe we get%
\begin{equation*}
\Theta _{l}^{0}=\left( \frac{3K_{2}}{K_{3}}\right) ^{l}\sum_{\alpha
_{1}+\alpha _{2}=l}u_{\alpha _{1},\alpha _{2}}\frac{u_{0,1}^{\alpha _{1}}}{%
\alpha _{1}!}\frac{\left( -u_{1,0}\right) ^{\alpha _{2}}}{\alpha _{2}!}%
\lambda ^{l},
\end{equation*}%
and therefore functions 
\begin{equation*}
I_{l}=\left( \frac{3K_{2}}{K_{3}}\right) ^{l}\sum_{\alpha _{1}+\alpha
_{2}=l}u_{\alpha _{1},\alpha _{2}}\frac{u_{0,1}^{\alpha _{1}}}{\alpha _{1}!}%
\frac{\left( -u_{1,0}\right) ^{\alpha _{2}}}{\alpha _{2}!}
\end{equation*}%
are invariants of order $l,$ for $l=4,...,k.$

Remark, that in order 4 we have two invariants $J_{4}$ and $I_{4}.$ It is
easy to check that there is the following relation between them:%
\begin{equation*}
J_{4}=2J_{0}J_{3}^{2}I_{4}-3J_{3}^{2}-\left( 3-\frac{6}{k}\right) J_{3}.
\end{equation*}

Moreover, the Christoffel symbols of the standard affine connection in the
invariant frame are invariants of the 4th order also.

Computing them we get:%
\begin{equation*}
\Gamma _{11}^{1}=J_{3}-\frac{J_{4}}{2J_{3}},\Gamma _{11}^{2}=\frac{J_{3}}{k}%
,\Gamma _{12}^{1}=1-k,\Gamma _{12}^{2}=0,\ \Gamma _{22}^{2}=1,\Gamma
_{22}^{1}=0.
\end{equation*}%
Invariant frame can be also written down in terms of Tresse derivatives:%
\begin{equation*}
\widehat{\lambda }=J_{4}\frac{d}{dJ_{3}},\widehat{\delta }=kJ_{0}\frac{d}{%
dJ_{0}}-kJ_{3}\frac{d}{dJ_{3}}.
\end{equation*}%
Therefore, the discussed above algebraic manifold $\Pi $ we'll get in the
following way.

Let $h$ be a binary form of degree $k$ and let 
\begin{equation*}
D_{h}:\mathbb{R}^{2}\rightarrow \mathbb{R}^{3}
\end{equation*}%
be the following rational mapping 
\begin{equation*}
D_{h}:x\mapsto \left( R_{1}=J_{3}\left( h\right) \left( x\right)
,R_{2}=h(x),U=J_{4}\left( h\right) \left( x\right) \right) ,
\end{equation*}%
where $\left( R_{1},R_{2},U\right) $ are the standard coordinates in $%
\mathbb{R}^{3}.$

Let $\Pi _{h}$ be the image of this mapping. Remark that functions $%
hJ_{3}\left( h\right) $ and $h^{2}J_{4}\left( h\right) $ are homogeneous
functions of degree zero therefore the image of surface $\Pi _{h}$ under the
rational map 
\begin{eqnarray*}
\chi &:&\mathbb{R}^{3}\rightarrow \mathbb{R}^{2}, \\
\chi &:&\left( R_{1},R_{2},U\right) \mapsto \left(
R_{1}R_{2},R_{2}^{2}U\right)
\end{eqnarray*}%
is a \textit{characteristic curve} $L_{h}\subset \mathbb{R}^{2},$ given by
equation%
\begin{equation*}
\rm{Res}\left( P_{1}-aQ_{1},P_{2}-bQ_{2}\right) =0,
\end{equation*}%
where $\left( P_{i},Q_{i}\right) $ are polynomials, say in $t=y/x,$ without
nontrivial common factor and such that 
\begin{equation*}
hJ_{3}\left( h\right) =\frac{P_{1}}{Q_{1}},\ h^{2}J_{4}\left( h\right) =%
\frac{P_{2}}{Q_{2}},
\end{equation*}%
$\left( a,b\right) $ coordinates on $\mathbb{R}^{2},$ and $\rm{Res}$ is
the resultant of polynomials.

Because $\Pi _{h}=\chi ^{-1}\left( L_{h}\right) $ we get the following
result by applying the above theorem.

\begin{theorem}
Two regular binary forms $h_{1}$ and $h_{2}$ of degree $k\geq 3$ are $G$%
-equivalent if and only if their characteristic curves coincide: $%
L_{h_{1}}=L_{h_{2}}.$
\end{theorem}

\section{Scalar differential operators}

\subsection{Jets of differential operators and natural invariants}

Denote by $\chi _{k}$: $\emph{Diff}_{k}(M)\rightarrow M$ \ the vector bundle
of scalar linear differential operators of the $k$-th order on manifold $M$.

Sections%
\begin{equation*}
S_{A}:M\rightarrow \emph{Diff}_{k}(M)
\end{equation*}

of this bundle will be identified with differential operators $A\in \mathbf{%
Diff}_{k}\left( M\right) $.

In this bundle we will use the following canonical local coordinates 
\begin{equation*}
\left( x_{1},..x_{n},u^{\alpha }\right) ,
\end{equation*}%
where $\left( x_{1},..x_{n}\right) $ are local coordinates on $M$ and $%
u^{\alpha }$ are fibre wise coordinates in bundle $\chi _{k}.$ Here $\alpha
=\left( \alpha _{1},...,\alpha _{n}\right) $ are multi indices of length $%
0\leq \left\vert \alpha \right\vert \leq k.$

In these coordinates the section $S_{A},$ that corresponds to operator 
\begin{equation*}
A=\sum_{\left\vert \alpha \right\vert \leq k}a^{\alpha }\left( x\right) 
\frac{\partial ^{\left\vert \alpha \right\vert }}{\partial x^{\alpha }},
\end{equation*}%
has the form%
\begin{equation*}
u^{\alpha }=a^{\alpha }\left( x\right) .
\end{equation*}

Exact sequence of modules (\ref{symbol seq}) gives us the exact sequence of
vector bundles%
\begin{equation}
0\rightarrow \chi _{k-1}\rightarrow \chi _{k}\rightarrow S^{k}\left( \tau
\right) \rightarrow 0.  \label{exact seq scalar}
\end{equation}

Denote by $\pi _{l}:J^{l}\left( \chi _{k}\right) \rightarrow M$ \ the vector
bundles of l-jets of sections of bundles $\chi _{k}$, or, in other words,
bundles of l-jets of

the k-th order scalar differential operators.

we'll denote by $[A]_{p}^{l}$ l-jets of operators at a point $p\in M$.

Bundles $\chi _{k},$ as well as bundles $\pi _{l}$ are natural in the sense
that the action of the diffeomorphism group $\mathcal{G}\left( M\right) $ is
lifted

to automorphisms of these bundles in the natural way:%
\begin{equation*}
\phi ^{\left( l\right) }:[A]_{p}^{l}\mapsto \lbrack \phi _{\ast }\left(
A\right) ]_{\phi \left( p\right) }^{l},
\end{equation*}%
for any diffeomorphism $\phi \in \mathcal{G}\left( M\right) .$

We'll study orbits (or equivalence) of scalar differential operators under
the action of the diffeomorphism group, but first of all we'll investigate
orbits of this action on the finite jet levels.

To this end we'll consider $\mathcal{G}\left( M\right) $-action on the jet
manifolds\\ $J^{l}\left( \chi _{k}\right) $, for k = 0, 1, 2, ...

Because $\mathcal{G}\left( M\right) $ acts in a transitive way on the base
manifold we fix point $p\in M$ and consider only jets of diffeomorphisms
which leave point $p$ fixed, i.e. consider the action of differential group $%
\mathbf{D}_{k+l}$ of order $\left( k+l\right) $ on the fibre $\mathbf{J}%
^{l}=J_{p}^{l}\left( \chi _{k}\right) .$

This is a linear action of the algebraic group $\mathbf{D}_{k+l}$ on vector
space $\mathbf{J}^{l}$ and due to Rosenliht theorem (see, for example, [\cite%
{Ros},\cite{KL2}]) the regular orbits of this action are separated by
rational on $\mathbf{J}^{l}$ invariants.

We call this invariants by \textit{natural differential invariants of order} 
$\leq l$, for the $k$-th order scalar linear differential operators. We also
say that an orbit $O$ is \textit{regular }if there are $m=\rm{codim}O$
rational invariants $I_{1},....,I_{m},$ which are independent in a
neighborhood of the orbit and such that 
\begin{equation*}
O=\left\{ I_{1}=c_{1},...,I_{m}=c_{m}\right\}
\end{equation*}
for some constants $c_{1},...,c_{m}$.

In other words, regularity of orbit $O$ means that the quotient space $%
\mathbf{J}^{l}/\mathbf{D}_{k+l}$ smooth at the point $O,$ and functions $%
I_{1},....,I_{m}$ could be considered as local coordinates in a neighborhood
of this point.

The universal construction, discussed above, could be applied for bundles $%
\chi _{k}$ (see, also, \cite{LY2},\cite{LY3}).

Namely, the standard reasons show that there is an unique total differential
operator of order $k:$%
\begin{equation*}
\square :C^{\infty }\left( J^{l}\left( \chi _{k}\right) \right) \rightarrow
C^{\infty }\left( J^{l+k}\left( \chi _{k}\right) \right) ,
\end{equation*}%
$l=0,1,...,$ such that 
\begin{equation}
j_{k+l}\left( S_{A}\right) ^{\ast }\left( \square \left( f\right) \right)
=A\left( j_{l}\left( S_{A}\right) ^{\ast }(f)\right) ,  \label{unidif}
\end{equation}%
for all functions $f\in C^{\infty }\left( J^{l}\left( \chi _{k}\right)
\right) $ and operators $A\in \mathbf{Diff}_{k}\left( M\right) .$

It is easy to see that in the standard jet-coordinates in the bundles $\pi
_{l}$ this operator has the following form%
\begin{equation*}
\square =\sum_{\left\vert \alpha \right\vert \leq k}u^{\alpha }\frac{%
d^{\left\vert \alpha \right\vert }}{dx^{\alpha }}.
\end{equation*}

The main property of this operator is its naturality:%
\begin{equation*}
\phi ^{\left( k+l\right) \ast }\circ \square =\square \circ \phi ^{\left(
l\right) \ast },
\end{equation*}%
for all diffeomorphisms $\phi .$

Let $J=\left( J_{1},..,J_{n}\right) $ be a set of natural differential
invariants. We say that they are in\textit{\ general position} if 
\begin{equation*}
\widehat{d}J_{1}\wedge ....\wedge \widehat{d}J_{n}\neq 0.
\end{equation*}%
Let $I$ be an invariant, then 
\begin{equation*}
\widehat{d}I=\sum_{i}I_{i}~\widehat{d}J_{i},
\end{equation*}%
for some rational functions $I_{i},$ which are called \textit{Tresse
derivatives. }

We'll denote them by $\frac{dI}{dJ_{i}}.$ Remark that they are invariants by
the construction, having, as a rule, higher order then invariant $I.$

Then the principle of n-invariants, formulated in \cite{ALV}, gives us the
following result.

\begin{theorem}
\label{n--invariants}Let's natural rational invariants $J_{1},..,J_{n}$ are
in general position. Then the field of all natural rational invariants is
generated by these invariants and the Tresse derivatives 
\begin{equation*}
\frac{d^{\left\vert \beta \right\vert }J_{\alpha }}{dJ^{\beta }}
\end{equation*}%
of invariants 
\begin{equation*}
J_{\alpha }=\square \left( J_{1}^{\alpha _{1}}\cdots J_{n}^{\alpha
_{n}}\right) ,
\end{equation*}%
with $0\leq \left\vert \alpha \right\vert \leq k.$
\end{theorem}

\begin{proof}
Assume that orders of all invariants $J_{i}$ less then $N-1.$ Then total
differentials $\widehat{d}J_{i}$ are linear independent in an open and dense
domain $\mathcal{O\subset }\mathbf{J}^{N}.$ Therefore, for almost all
differential operators $A\in \mathbf{Diff}_{k}\left( M\right) ,$ their
values $J_{i}\left( A\right) $ give us local coordinates in a neighborhood
of the point $p\in M.$ Then functions $J_{\alpha }\left( A\right) $ allow us
to find all coefficients of the operator $A$ in this coordinates and values $%
\frac{d^{\left\vert \beta \right\vert }J_{\alpha }}{dJ^{\beta }}\left(
A\right) $ give us all derivatives of these coefficients. Therefore, any
rational natural invariant $J$ is a rational function of invariants $J_{i}$
and $\frac{d^{\left\vert \beta \right\vert }J_{\alpha }}{dJ^{\beta }},$ by
the definition of invariants.
\end{proof}

We apply this construction for cases, where we have at least $n$ independent
invariants of order zero. To this end we consider the symbol mapping $%
\rm{smbl}:\mathbf{J}^{0}\rightarrow S^{k}T_{p}\rightarrow 0.$ This
mapping commutes with the action of the diffeomorphism group and therefore $%
\rm{smbl}^{\ast }\left( J\right) $ is a natural invariant of order zero
for differential operators if $J$ is a $\rm{GL}\left( T_{p}\right) $%
-invariant of homogeneous $k$-forms.

We'll consider cases, when $n\geq 2$ and $k\geq 3.$ Then the above
description of invariants of homogeneous $k$-forms shows that stationary Lie
algebras of regular forms are trivial and therefore codimensions of regular
orbits are 
\begin{equation*}
c\left( n,k\right) =\binom{n+k-1}{k}-n^{2}.
\end{equation*}%
It is easy to check that $c\left( n,k\right) \geq n$ for all $n\geq 2,k\geq
3,$ with the exception for the following three cases:%
\begin{eqnarray*}
n &=&2,k=3; \\
n &=&2,k=4; \\
n &=&3,k=3.
\end{eqnarray*}

\begin{theorem}
In non exceptional cases the field of natural invariants is generated by
invariants of the zero order.
\end{theorem}

\subsection{ Differential operators of general type}

We say that an operator $A\in \mathbf{Diff}_{k}\left( M\right) $ has a 
\textit{general type }at a point $q\in M$ \ if there are differential
invariants in general position such that 
\begin{equation}\label{geneposition}
dJ_{1}\left( A\right) \wedge \cdots \wedge dJ_{n}\left( A\right) \neq 0,
\end{equation}%
at the point $q.$

Then, as we have seen, functions $J_1(A) ,\cdots ,J_{n}\left(
A\right) $ are local coordinates, we call them \textit{natural coordinates}, and
functions $J_{\alpha }(A) = F_{\alpha }\left( J_{1}(A),\cdots, J_{n}(A) \right)$ allow us to find all
coefficients of the operator $A$ in terms of the local coordinates.

Therefore, the relations $J_{\alpha }=F_{\alpha }\left( J_{1},\cdots
,J_{n}\right) $ completely define differential operators locally up to
diffeomorphism.

We say also that an operator $A\in \mathbf{Diff}_{k}\left( M\right) $ has
symbol of \textit{general type }at a point $q\in M$ if the symbol of
operator $A$ is regular and there are $\rm{GL}\left( T_{q}\right) $%
-invariants $J_{1},...,J_{n}$ of the $\rm{GL}\left( T_{q}\right) $%
-action on $S^{k}T_{q},$ i.e. differential invariants of the zero order,
which are in general position.

In the last case, relations $J_{\alpha }=F_{\alpha }\left( J_{1},\cdots
,J_{n}\right) $ also completely define differential operators.

Summarizing, we get the following result.\newpage

\begin{theorem} $\phantom{\alpha}$
\begin{enumerate}
\item The relations $J_{\alpha }=F_{\alpha }\left( J_{1},\cdots
,J_{n}\right) $, $0\leq \left\vert \alpha \right\vert \leq k,$ for
differential invariants $J_{1},\cdots ,J_{n}$ in general position, locally
define general type differential operators up to local diffeomorphisms.

\item The same relations for differential operators, having symbols of
general type, where $J_{1},\cdots ,J_{n}$ are $\rm{GL}\left(
T_{q}\right) $-invariants of symbols being in general position, locally
define differential operators up to local diffeomorphisms. In particular,
these invariants of the zero order generate the field of rational natural
operators.
\end{enumerate}
\end{theorem}

\begin{remark}
It follows from the Rosenliht theorem \cite{Ros},\cite{KL2}, that the
field of rational natural invariants separates regular orbits in the jet
spaces of differential operators.
\end{remark}

\subsection{Symbols of general type}

Similar to differential operators, we say that a symbol $\sigma \in \Sigma
_{k}\left( M\right) $ has a\textit{\ general type} at a point $a\in M$ there
are differential invariants of symbols of $k$-th degree, say $J_{1},..,J_{n}$
such that condition (\ref{geneposition}) holds.

Functions $J_{1}\left( \sigma \right) ,\cdots ,J_{n}\left( \sigma \right) ,$
as above, are local coordinates in a neighborhood of the point $a\in M,$ and
we call them also \textit{natural. \ }

Let\textit{\ }%
\begin{equation*}
J_{\alpha }=\left( \widehat{dJ}_{1}^{\alpha _{1}}\cdot \cdots \cdot \widehat{%
dJ}_{n}^{\alpha _{n}}\right) ^{\alpha }\rfloor \nu _{k}
\end{equation*}%
be differential invariants, for multi indices $\alpha =\left( \alpha
_{1},...,\alpha _{n}\right) ,$ $\left\vert \alpha \right\vert =k.$

Then functions $J_{\alpha }\left( \sigma \right) =F_{\alpha }\left(
J_{1}\left( \sigma \right) ,\cdots ,J_{n}\left( \sigma \right) \right) $
completely defines symbol $\sigma $ in the natural coordinates and we get
the following.

\begin{theorem}
The relations $J_{\alpha }=F_{\alpha }\left( J_{1},\cdots ,J_{n}\right) $, $%
\left\vert \alpha \right\vert =k,$ for differential invariants $J_{1},\cdots
,J_{n}$ of symbols locally define general type symbols up to local
diffeomorphisms. The principle of n-invariants (\ref{n--invariants}) is also
valid for symbols.
\end{theorem}

\begin{remark}
In this paper, for obvious reasons, we work only with contravariant
symmetric tensors but all results on classification of covariant symmetric
tensors, as well as their $\rm{GL}$-orbits, are also valid.
\end{remark}

\section{Differential operators, acting in line bundles}

Let $V$ be a vector space of dimension $n$ $\left( =\dim M\right) $ and let\\$%
\mathcal{\varpi \subset }S^{k}V$ be a regular $\rm{GL}\left( V\right) $%
-orbit. We say that an operator $A\in \mathbf{Diff}_{k}\left( \xi \right) ,$
acting in a linear bundle $\xi :E\left( \xi \right) \rightarrow M,$ has a 
\textit{constant type} $\mathcal{\varpi }$ if for any point $q\in M$ and any
isomorphism $\phi :T_{q}M\rightarrow V$ the image of the symbol $\phi _{\ast
}\left( \rm{smbl}_{k}\left( A\right) \right) \in S^{k}V$ belongs to $%
\mathcal{\varpi }$ .

In this case, we are not able to use $\rm{GL}$-invariants of symbols as
natural coordinates as it was done in the above theorem.

On the other hand, we will see that the description of such operators (and
their natural differential invariants) based on existence of connections on
the base manifold $M$ as well as in the linear bundle $\xi ,$ which are
naturally related to the differential operators. These connections, as we
have seen, allow us to establish isomorphism between differential operators
and their total symbols and therefore realize "the dream" of wave mechanics
on existence of natural procedure for quantization of differential operators.

Remark that for operators of the second order it was done in \cite{LY2}
and for operators of the third order over 2-dimensional manifolds in \cite%
{LY3}.

\subsection{Wagner and Chern connections}

Let $\sigma \in \Sigma _{k}\left( M\right) $ be a symbol of a \textit{%
constant type} $\mathcal{\varpi },$ and let $k\geq 3$ and $n\geq 2.$ Then,
due to (\ref{ISOTROPY}), for any point $p\in M$ there is a neighborhood $%
\mathcal{O}_{p}$ and unique linear isomorphisms $A_{\,p,p^{\prime }}:$ $%
T_{p}M\rightarrow T_{p^{\prime }}M,$ for all $p^{\prime }\in \mathcal{O}%
_{p}, $ such that $A_{p,p^{\prime }}^{\ast }\left( \sigma _{p^{\prime
}}\right) =\sigma _{p}.$ Therefore, there is and unique linear connection $%
\nabla ^{\sigma }$(we call it \textit{Wagner connection}, see \cite{Wag},%
\cite{LY3}) on manifold $M$ such that isomorphisms $A_{\,p,p^{\prime }}$ are
operators of the parallel transport along $\nabla ^{\sigma }.$ This
connection preserves the symbol, i.e.%
\begin{equation}
\nabla _{X}^{\sigma }\left( \sigma \right) =0,  \label{WagnerConnection}
\end{equation}%
for all vector fields $X$ on $M.$

On the other hand, if symbol $\sigma $ admits such connection then it has a
constant type. Summarizing we get the following.

\begin{theorem} $\phantom{a}$
\begin{enumerate}
\item A regular symbol $\sigma \in \Sigma _{k}\left( M\right) ~$of constant
type has a unique Wagner connection (\ref{WagnerConnection}), if $k\geq
3,n\geq 2,$

\item The Wagner connection has the trivial curvature tensor, $R^{\sigma }$ $%
=0$.
\end{enumerate}
\end{theorem}

Keeping in mind differential equations, but not differential operators, we
consider conformal classes $[\sigma ]=\left\{ f\sigma ,f\in \mathcal{F}%
\left( M\right) \right\} $ of symbols.

Then the Wagner connections $\nabla ^{\sigma }$ and $\nabla ^{f\sigma }$
preserve the conformal class $[\sigma ]$ and therefore their covariant
differential satisfy the relation%
\begin{equation}
d_{\nabla ^{f\sigma }}-d_{\nabla ^{\sigma }}=\omega \otimes \rm{Id},
\label{WagDif}
\end{equation}%
for some differential 1-form $\omega .$

On the other hand relation (\ref{WagnerConnection}) shows that%
\begin{equation*}
\nabla _{X}^{f\sigma }\left( \sigma \right) -\nabla _{X}^{\sigma }\left(
\sigma \right) =\nabla _{X}^{f\sigma }\left( f^{-1}f\sigma \right) =X\left(
f^{-1}\right) f\sigma =-\frac{X\left( f\right) }{f}\sigma .
\end{equation*}

Therefore, the differential $1$-form in relation (\ref{WagDif}) equals:%
\begin{equation*}
\omega =-d\ln \left( \left\vert f\right\vert \right) .
\end{equation*}

Let $T^{\sigma }\left( X,Y\right) $ be the torsion tensor of the Wager
connection and let $\theta ^{\sigma }$ be the torsion form of this
connection,i.e.%
\begin{equation*}
\theta ^{\sigma }\left( X\right) =\rm{Tr}\left( Y\rightarrow T^{\sigma
}\left( X,Y\right) \right) .
\end{equation*}

Then, it is easy to see that,%
\begin{equation}
T^{f\sigma }\left( X,Y\right) -T^{\sigma }\left( X,Y\right) =\omega \left(
X\right) Y-\omega \left( Y\right) X  \label{TorDif}
\end{equation}%
and%
\begin{equation}
\theta ^{f\sigma }\left( X\right) -\theta ^{\sigma }\left( X\right) =\left(
n-1\right) \omega \left( X\right) .  \label{TorFormDif}
\end{equation}

Let denote by $\widetilde{\nabla }^{\sigma }$ the following deformation of
the Wagner connection:%
\begin{equation*}
\widetilde{\nabla }_{X}^{\sigma }=\nabla _{X}^{\sigma }+\lambda \ \theta
^{\sigma }\left( X\right) ,
\end{equation*}%
for some $\lambda \in \mathbb{R}.$

Then, due to (\ref{WagDif}) and (\ref{TorFormDif}), we have 
\begin{equation*}
\widetilde{\nabla }_{X}^{f\sigma }-\widetilde{\nabla }_{X}^{\sigma }=\omega
\left( X\right) +\lambda \left( n-1\right) \omega \left( X\right) =\left(
1+\lambda \left( n-1\right) \right) \omega \left( X\right) .
\end{equation*}%
Therefore, the following connection, we call it \textit{Chern connection},%
\begin{equation*}
\nabla _{X}^{[\sigma ]}=\nabla _{X}^{\sigma }-\frac{1}{n-1}\theta ^{\sigma
}\left( X\right) ,
\end{equation*}%
does not depend on representative of the conformal class $[\sigma ]$ and the
corresponding parallel transforms preserve this class.

Summarizing, we get the following result.\newpage

\begin{theorem}$\phantom{\alpha}$
\begin{enumerate}
\item For any conformal class $[\sigma ]$ of regular constant type symbol $%
\sigma $ of degree $k\geq 3,$ there is and unique Chern connection $\nabla
^{\lbrack \sigma ]},$ that preserves the conformal class%
\begin{equation*}
\nabla _{X}^{[\sigma ]}\left( \sigma \right) =-\frac{1}{n-1}\theta ^{\sigma
}\left( X\right) \sigma .
\end{equation*}

\item The torsion form of the Chern connection equals zero.
\end{enumerate}
\end{theorem}

\begin{remark}
The Chern connection is a torsion free connection in dimension $n=2.$
\end{remark}

\subsection{Group-type symbols}

Let $M$ \ be a connected and simply connected manifold and let $\nabla
^{\sigma }$ be the Wagner connection with torsion tensor $T^{\sigma }.$
We'll assume that this connection is complete and its torsion tensor is
parallel%
\begin{equation}
d_{\nabla ^{\sigma }}\left( T^{\sigma }\right) =0.  \label{paralleltorsion}
\end{equation}

Then, it is easy to check that the vector space $\mathfrak{g}^{\sigma }$ of
all parallel vector fields on $M,$ having dimension $n=\dim M,$ is a Lie
algebra with respect to bracket%
\begin{equation*}
X,Y\in \mathfrak{g}^{\sigma }\rightarrow T^{\sigma }\left( X,Y\right) \in 
\mathfrak{g}^{\sigma }.
\end{equation*}

Moreover, relations (\ref{paralleltorsion}) and $R^{\sigma }=0$ implies (see 
\cite{KN}) the following:

\begin{enumerate}
\item $M$ is an analytic manifold,

\item $\sigma $ is an analytic symbol, and

\item $\nabla ^{\sigma }$ is an analytic linear connection.
\end{enumerate}

Let $G^{\sigma }$ be a connected and simply connected Lie group with Lie
algebra $\mathfrak{g}^{\sigma },$ and let $\widetilde{\sigma }\in \Sigma
_{k}\left( G^{\sigma }\right) $ be a $G^{\sigma }$-invariant symbol of the
constant type $\mathcal{\varpi }.$ The following result gives us a
classification of such type symbols and it follows from the application of
Theorem 7.8 in \cite{KN} to our case.

\begin{theorem}
Let $M$ be a connected and simply connected manifold and let $\sigma $ be a
regular symbol of degree $k\geq 3,$ having constant type $\mathcal{\varpi }.$
Let also the Wagner connection $\nabla ^{\sigma }$ be complete and the
torsion tensor $T^{\sigma }$ be parallel. Then any linear isomorphism $%
\digamma :T_{p}M\rightarrow T_{e}G^{\sigma },$ such that $\digamma \left(
\sigma _{p}\right) =\widetilde{\sigma }_{e},$ could be extended to affine
diffeomorphism $F:M\rightarrow G^{\sigma },$ such that $F_{\ast }\left(
\sigma \right) =\widetilde{\sigma }$ and $F_{\ast ,p}=\digamma .$
\end{theorem}

\subsection{Connections, associated with differential operators of\\ constant
type}

In this section we'll consider operators $A\in \mathbf{Diff}_{k}\left( \xi
\right) ,$ acting in a line bundles $\xi :E\left( \xi \right) \rightarrow M$
and having a constant type\textit{. }We'll show\textit{\ }that, under some
generality conditions, these operators generate linear connections in the
bundle $\xi ,$ which are in a natural way associated with with operators.

As above, we'll restrict ourselves by the case $k\geq 3,n\geq 2,$ although
the case of ordinary differential operators, $n=1,$ we'll considered
separately, as an example.

Thus, let $A\in \mathbf{Diff}_{k}\left( \xi \right) $ be an operator of the
constant type and let $\sigma =\sigma _{k}=\rm{smbl}_{k}(A)\in \Sigma
_{k}\left( M\right) $ be its symbol. We'll assume that $\sigma $ is regular
and denote by $\nabla ^{\sigma }$ the Wagner connection and by $\theta
^{\sigma }$ the torsion form of this connection.

Let $\nabla $ be a linear connection in the line bundle and let $Q_{\nabla
,\nabla ^{\sigma }}$ be the quantization defined by these two connections,
and let 
\begin{equation*}
A=Q_{\nabla ,\nabla ^{\sigma }}\left( \sum_{i}\sigma _{i,\nabla }\right) ,
\end{equation*}%
be its decomposition, where $\sigma _{k,\nabla }=\sigma ,$ and 
\begin{equation*}
\sigma _{k-1,\nabla }=\rm{smbl}_{k-1}(A-Q_{\nabla ,\nabla
^{\sigma }}\left( \sigma \right) ).
\end{equation*}%
Let $\nabla ^{\prime }$ be another linear connection in the line bundle and
let 
\begin{equation*}
d_{\nabla ^{\prime }}-d_{\nabla }=\theta \otimes \rm{Id},
\end{equation*}%
for some differential form $\theta \in \Omega ^{1}\left( M\right) .$

Then, due to definition (\ref{Quant}), we have 
\begin{equation}
\rm{smbl}_{k-1}\left( Q_{\nabla ^{\prime },\nabla ^{\sigma
}}\left( \sigma \right) -Q_{\nabla ,\nabla ^{\sigma }}\left( \sigma \right)
\right) =\theta \rfloor \sigma \in \Sigma _{k-1}\left( M\right) ,
\label{QconnDiff}
\end{equation}%
and therefore%
\begin{equation}
\sigma _{k-1,\nabla ^{\prime }}-\sigma _{k-1,\nabla }=\theta \rfloor \sigma .
\label{sigmak-1}
\end{equation}

We'll say that a regular symbol $\sigma $ is \textit{Wagner regular} if
quadratic symbol 
\begin{equation*}
g_{W}=\left( \theta ^{\sigma }\right) ^{k-2}\rfloor \sigma \in \Sigma
_{2}\left( M\right)
\end{equation*}%
is non degenerated.

\begin{theorem}
Let's a differential operator $A\in \mathbf{Diff}_{k}\left( \xi \right) $
has a constant type and its symbol is the Wagner regular. Then there exists
and unique a linear connection $\nabla ^{A}$ in the line bundle $\xi $ such
that 
\begin{equation}
\left( \theta ^{\sigma }\right) ^{k-2}\rfloor \sigma _{k-1,\nabla ^{A}}=0.
\label{conncond}
\end{equation}
\end{theorem}

\begin{proof}
Due to (\ref{sigmak-1}), we have $\nabla ^{\prime }=\nabla ^{A}$ if and only
if 
\begin{equation*}
\theta \rfloor g_{W}+\left( \theta ^{\sigma }\right) ^{k-2}\rfloor \sigma
_{k-1,\nabla }=0,
\end{equation*}%
and the last equation has a unique solution $\theta $ if $g_{W}$ non
degenerated.
\end{proof}

\begin{remark}
Connections $\nabla ^{A}$ and the Wagner connection $\nabla ^{\rm{%
smbl}_{k}\left( A\right) }$ are natural in the sense that 
\begin{equation*}
\phi _{\ast }\left( \nabla ^{A}\right) =\nabla ^{\phi _{\ast }(A)},\phi
_{\ast }\left( \nabla ^{\rm{smbl}_{k}\left( A\right) }\right)
=\nabla ^{\phi _{\ast }(\rm{smbl}_{k}\left( A\right) )},
\end{equation*}%
for any $\phi \in \mathbf{Aut}(\xi ).$
\end{remark}

\subsection{Connections, associated with differential equations}

In order to study homogeneous differential equations, associated with
differential operators we'll study conformal classes of differential
operators $[A]=\left\{ \left. fA\right\vert \ f\in \mathcal{F}\left(
M\right) \right\} $ and geometrical structures associated with them. All
operators in this section are assumed to be regular and constant type.\ 

First of all we'll change connections and will consider Chern connections $%
\nabla ^{C}\overset{\text{def}}{=}\nabla ^{\lbrack \sigma ]}$ associated
with conformal classes $[\sigma ]$ of the symbols instead of the Wagner $%
\nabla ^{\sigma }$ ones.

As we have seen (\ref{TorFormDif}) the torsion form $\theta ^{\sigma }$ is
not invariant of the conformal class but its differential $\omega
^{C}=d\theta ^{\sigma }$ does.

In what follows we'll need some constructions from the linear algebra. To
this end we'll fix a point on $M$ and denote by $T$ the tangent space at the
point.

Let 
\begin{equation}
\widehat{\omega }:T\rightarrow T^{\ast }  \label{ChernOperator}
\end{equation}%
be the linear operator, defined by $\omega ^{C},$ i.e. 
\begin{equation*}
\left\langle \widehat{\omega }\left( X\right) ,Y\right\rangle =\omega
^{C}\left( X,Y\right) ,
\end{equation*}%
for all vectors $X,Y\in T.$

Let 
\begin{equation*}
\widehat{\omega _{l}}:S^{l}T\rightarrow S^{l}T^{\ast },
\end{equation*}%
be its $l$-th symmetric power for $l=1,2,..,$%
\begin{equation*}
\widehat{\omega _{l}}\left( X_{1}\cdot \cdots \cdot X_{l}\right) =\widehat{%
\omega }\left( X_{1}\right) \cdot \cdots \cdot \widehat{\omega }\left(
X_{l}\right) ,
\end{equation*}%
where $X_{i}\in T,$ $i=1,...,l.$

Remark that $S^{l}T^{\ast }\simeq \left( S^{l}T\right) ^{\ast }$ and
therefore operators $\widehat{\omega _{l}}$ defines bilinear forms $\omega
_{l}$ on $S^{l}T$ as follows 
\begin{equation*}
\omega _{l}\left( X_{1}\cdot \cdots \cdot X_{l},Y_{1}\cdot \cdots \cdot
Y_{l}\right) =\left\langle \widehat{\omega }\left( X_{1}\right) \cdot \cdots
\cdot \widehat{\omega }\left( X_{l}\right) ,Y_{1}\cdot \cdots \cdot
Y_{l}\right\rangle .
\end{equation*}%
Moreover, we have $\left( \widehat{\omega }\right) ^{\ast }=-\widehat{\omega 
}$ and therefore $\left( \widehat{\omega _{l}}\right) ^{\ast }=\left(
-1\right) ^{l}\widehat{\omega }_{l},$ i.e. forms $\omega _{l}$ are symmetric
or skew symmetric when $l$ is even or odd. They also are non degenerated if
the initial form $\omega ^{C}$ does.

As we also have seen (\ref{sigmak-1}) change of connection in the bundle
leads us to change of sub symbol $\sigma _{k-1}$ in the following way 
\begin{equation*}
\sigma _{k-1}\longmapsto \sigma _{k-1}+\theta \rfloor \sigma ,
\end{equation*}
and therefore defines an $T^{\ast }$-action in $S^{k-1}T.$

Denote by 
\begin{equation*}
L_{[\sigma ]}=\left\{ \left. \theta \rfloor \sigma \ \ \right\vert \theta
\in T^{\ast }\right\} \subset S^{k-1}T
\end{equation*}%
a subspace in $S^{k-1}T.$

This subspace has dimension $n=\dim T$ and depends on the conformal class of
the symbol, and orbits $O\left( \sigma _{k-1}\right) $ of tensors $\sigma
_{k-1}\in S^{k-1}T$ under the $T^{\ast }$-action are affine subspaces 
\begin{equation*}
O\left( \sigma _{k-1}\right) =\sigma _{k-1}+L_{[\sigma ]}.
\end{equation*}

Let $L_{[\sigma ]}^{0}\subset S^{k-1}T$ be the annihilator of the image $%
\widehat{\omega _{l}}\left( L_{[\sigma ]}\right) .$

\begin{proposition}
Let the following conditions hold:

\begin{enumerate}
\item 
\begin{equation}
L_{[\sigma ]}\cap \ker \widehat{\omega }_{k-1}=0.  \label{DimCond}
\end{equation}

\item 
\begin{equation}
L_{[\sigma ]}\cap L_{[\sigma ]}^{0}=0,  \label{TransverseCond}
\end{equation}
\end{enumerate}

Then for any tensor $\sigma _{k-1}\in S^{k-1}T$ there is and unique tensor $%
\sigma _{k-1}^{0}\in L_{[\sigma ]}^{0}\cap O\left( \sigma _{k-1}\right) $
and covector $\theta \in T^{\ast }$ such that 
\begin{equation}
\sigma _{k-1}=\sigma _{k-1}^{0}+\theta \rfloor \sigma .
\label{EquationConnectionForm}
\end{equation}
\end{proposition}

\begin{proof}
Condition (\ref{DimCond}) shows that $\dim \left( \widehat{\omega }%
_{k-1}\left( L_{[\sigma ]}\right) \right) =n$ and therefore $\rm{codim}%
\left( L_{[\sigma ]}^{0}\right) =n,$ next condition (\ref{TransverseCond})
shows that subspaces $L_{[\sigma ]}$ and $L_{[\sigma ]}^{0}$ are
transversal. Therefore any orbit $O\left( \sigma _{k-1}\right) $ has a
unique intersection with $L_{[\sigma ]}^{0}.$
\end{proof}

To reformulate transversality condition (\ref{TransverseCond}) let's take a
tensor $\gamma \in L_{[\sigma ]}\cap L_{[\sigma ]}^{0}.$ Then $\gamma \in
L_{[\sigma ]}$ implies that $\gamma =\alpha \rfloor \sigma ,$ for some
covector $\alpha \in T^{\ast },$ and $\gamma \in L_{[\sigma ]}^{0}$ means
that $\omega _{k-1}\left( \gamma ,L_{[\sigma ]}\right) =0,$ or $\omega
_{k-1}\left( \alpha \rfloor \sigma ,\beta \rfloor \sigma \right) =0,$ for
all $\beta \in T^{\ast }.$ Denote by $\omega _{\sigma }$ the following 
\textit{characteristic bivector} (symmetric, when $k$ is odd, and skew
symmetric for even $k$)%
\begin{equation}
\omega _{\sigma }\left( \alpha ,\beta \right) =\omega _{k-1}\left( \alpha
\rfloor \sigma ,\beta \rfloor \sigma \right) .  \label{CharBivector}
\end{equation}

\begin{proposition}
Subspaces $L_{[\sigma ]}$ and $L_{[\sigma ]}^{0}$ are transversal if (\ref%
{DimCond}) holds and \textit{characteristic bivector }$\omega _{\sigma }$
non degenerated.
\end{proposition}

\begin{corollary}
Let $k$ be even and conditions (\ref{DimCond}, \ref{TransverseCond}) are
valid. Then $\dim M$ is even too.
\end{corollary}

\begin{definition}
We say that a regular and constant type operator $A\in \mathbf{Diff}%
_{k}\left( \xi \right) $ is Chern regular if conditions (\ref{DimCond}, \ref%
{TransverseCond}) hold.
\end{definition}

\begin{theorem}
Let $A\in \mathbf{Diff}_{k}\left( \xi \right) $ be Chern regular operator.
Then there is and unique linear connection $\nabla ^{\lbrack A]}$ in the
line bundle $\xi ,$ depending on conformal class $[A],$ and such that the
subsymbol $\rm{smbl}_{k-1}\left( A\right) ,$ defining by this
connection, belongs to $L_{[\rm{smbl}_{k}\left( A\right) ]}^{0}.$
\end{theorem}

\begin{proof}
Let $\nabla $ be a connection in the line bundle and let $\sigma _{k-1}$ be
subsymbol of operator $A.$ Then $d_{\nabla ^{\lbrack A]}}-d_{\nabla }=\theta
\otimes \rm{Id},$ and sybsymbol of $A,$ defining by $\nabla ^{\lbrack
A]},$ belong to $L_{[\rm{smbl}_{k}\left( A\right) ]}^{0}$ if and only
if 
\begin{equation*}
\sigma _{k-1}-\theta \rfloor \sigma \in L_{[\rm{smbl}_{k}\left(
A\right) ]}^{0}.
\end{equation*}%
As we have seen this condition uniquely defines covector $\theta $ if
operator $A$ is Chern regular.

If we take another operator $fA$ from the conformal class then $\sigma
_{k-1} $ and $\rm{smbl}_{k}\left( A\right) $ are multiplied by $f$ \
but $\theta $ will not be changed.
\end{proof}

\begin{remark}
Connections $\nabla ^{\lbrack A]}$ and the Chern connection $\nabla
^{\lbrack \rm{smbl}_{k}\left( A\right) ]}$ are natural in the
sense that 
\begin{equation*}
\phi _{\ast }\left( \nabla ^{\lbrack A]}\right) =\nabla ^{\lbrack \phi
_{\ast }(A)]},\phi _{\ast }\left( \nabla ^{\lbrack \rm{smbl}%
_{k}\left( A\right) ]}\right) =\nabla ^{\lbrack \phi _{\ast }(%
\rm{smbl}_{k}\left( A\right) )]},
\end{equation*}%
for all automorphisms $\phi \in \mathbf{Aut}(\xi ).$
\end{remark}

\subsection{Example: Ordinary differential operators}

The case of ordinary differential operators in many aspect exceptional from
the point of view represented in this paper.

In this section we'll discus it in more details.

Let 
\begin{equation*}
A=a_{k}\partial ^{k}+\cdots +a_{1}\partial +a_{0},
\end{equation*}%
be a scalar ordinary differential operator, where $a_{i}=a_{i}\left(
x\right) ,$ $\partial =d/dx,$ $M=\mathbb{R}.$

This operator has the constant type if and only if $a_{k}\neq 0.$

Assume that a function $\Gamma \left( x\right) $ is a Christoffel
coefficient of a linear connection $\nabla $ on $M:$%
\begin{equation*}
\nabla _{\partial }\left( \partial \right) =\Gamma \partial .
\end{equation*}

Then, the quantization $Q:\Sigma _{k}\rightarrow \mathbf{Diff}_{k}\left(
M\right) ,$ associated with this connection, acts in the following way:%
\begin{equation*}
Q\left( \partial ^{k}\right) \left( f\right) =w^{-k}\left( w\partial
_{x}-\Gamma w^{2}\partial _{w}\right) ^{k}\left( f\right) ,
\end{equation*}%
for any function $f=f(x).$

In particular,%
\begin{eqnarray*}
Q\left( 1\right) &=&1, \\
Q\left( \partial \right) &=&\partial , \\
Q\left( \partial ^{2}\right) &=&\partial ^{2}-\Gamma \partial , \\
Q\left( \partial ^{3}\right) &=&\partial ^{3}-3\Gamma \partial ^{2}+\left(
\Gamma ^{2}-\Gamma ^{\prime }\right) \partial \text{.}
\end{eqnarray*}

This connection is the Wagner connection if $\nabla _{\partial }\left(
a_{k}\partial ^{k}\right) =0,$ or 
\begin{equation*}
\Gamma =-\frac{a_{k}^{\prime }}{ka_{k}}
\end{equation*}

If operator $A$ has order 2 and $\nabla $ is the Wagner connection we get 
\begin{equation*}
A=Q\left( \sigma _{2}+\sigma _{1}+\sigma _{0}\right) ,
\end{equation*}%
where%
\begin{eqnarray}
\sigma _{2} &=&a_{2}\partial ^{2},  \label{sigmadiffeo} \\
\sigma _{1} &=&\left( a_{1}-\frac{a_{2}^{\prime }}{2}\right) \partial , 
\notag \\
\sigma _{0} &=&a_{0}.  \notag
\end{eqnarray}

For operators of the third order the corresponding Wagner connections are of
the form%
\begin{equation*}
\Gamma =-\frac{a_{3}^{\prime }}{3a_{3}},
\end{equation*}%
and 
\begin{equation*}
A=Q\left( \sigma _{3}+\sigma _{2}+\sigma _{1}+\sigma _{0}\right) ,
\end{equation*}%
where%
\begin{eqnarray}
\sigma _{3} &=&a_{3}\partial ^{3},  \label{sigmadiffeo3} \\
\sigma _{2} &=&\left( a_{2}-a_{3}^{\prime }\right) \partial ^{2},  \notag \\
\sigma _{1} &=&\left( \frac{4a_{3}^{\prime 2}}{9a_{3}}-\frac{%
a_{2}a_{3}^{\prime }}{3a_{3}}-\frac{a_{3}^{\prime \prime }}{3}\right)
\partial ,\   \notag \\
\sigma _{0} &=&a_{0}.  \notag
\end{eqnarray}%
In the case, when $A\in \mathbf{Diff}_{k}\left( \xi \right) $ and $\xi $ is
a line bundle over $M=\mathbb{R}$ with a linear connection $\nabla ^{\xi }$
defined by a differential 1-form $\theta \left( x\right) dx,$ the
quantization $Q:\Sigma _{k}\left( M\right) \rightarrow \mathbf{Diff}%
_{k}\left( \xi \right) ,$ associated with the connection $\nabla $ on $M$
and the connection $\nabla ^{\xi },$ has the form 
\begin{equation*}
Q\left( \partial ^{k}\right) \left( f\right) =w^{-k}\left( w\partial
_{x}-\Gamma w^{2}\partial _{w}+w\theta \right) ^{k}\left( f\right) ,
\end{equation*}%
and for low orders has the form%
\begin{eqnarray*}
Q\left( 1\right) &=&1,\  \\
Q\left( \partial \right) &=&\partial +\theta ,\  \\
Q\left( \partial ^{2}\right) &=&\partial ^{2}+(2\theta -\Gamma )\partial
+\theta ^{\prime }-\theta \Gamma +\theta ^{2}, \\
Q\left( \partial ^{3}\right) &=&\partial ^{3}+3\left( \theta -\Gamma \right)
\partial ^{2}+\left( 2\Gamma ^{2}-6\theta \Gamma +3\theta ^{2}-\Gamma
^{\prime }+3\theta ^{\prime }\right) \partial + \\
&&\theta ^{\prime \prime }+3\left( \theta -\Gamma \right) \theta ^{\prime
}+\left( \left( \Gamma -\theta \right) \left( 2\Gamma -\theta \right)
-\Gamma ^{\prime }\right) \theta .
\end{eqnarray*}%
In the case when $\nabla $ is the Wagner connection for given operator $A$
we define the associated connection $\nabla ^{A}$ by a requirement slightly
different from the above. Namely, we'll require that%
\begin{equation}
\left( \theta dx\right) \rfloor \sigma _{k}=\sigma _{k-1}.
\label{A-ordconnection}
\end{equation}%
Thus, for the second order operators,%
\begin{equation*}
A=a_{2}\partial ^{2}+a_{1}\partial +a_{0},
\end{equation*}

we have 
\begin{eqnarray*}
\sigma _{2} &=&a_{2}\partial ^{2}, \\
\sigma _{1} &=&\left( a_{1}+a_{2}\left( \Gamma -2\theta \right) \right)
\partial .
\end{eqnarray*}%
Therefore, in this case we get%
\begin{equation*}
\Gamma =-\frac{a_{2}^{\prime }}{2a_{2}},\theta =\frac{2a_{1}-a_{2}^{\prime }%
}{8a_{2}},
\end{equation*}%
and the invariant quantization $Q_{\nabla _{W},\nabla ^{A}}$ has the form%
\begin{eqnarray*}
Q\left( 1\right) &=&1, \\
Q\left( \partial \right) &=&\partial +\frac{2a_{1}-a_{2}^{\prime }}{8a_{2}},
\\
Q\left( \partial ^{2}\right) &=&\partial ^{2}+\frac{2a_{1}+a_{2}^{\prime }}{%
4a_{2}}\partial -\frac{a_{2}^{\prime \prime }+2a_{1}^{\prime }}{8a_{2}}+%
\frac{5}{64}\left( \frac{a_{2}^{\prime }}{a_{2}}\right) ^{2}+\frac{%
a_{1}^{2}-3a_{1}a_{2}^{\prime }}{16a_{2}^{2}},
\end{eqnarray*}%
and%
\begin{eqnarray}
\sigma _{2} &=&a_{2}\partial ^{2},  \label{sigmaauto2} \\
\sigma _{1} &=&\left( \frac{a_{1}}{2}-\frac{a_{2}^{\prime }}{4}\right)
\partial ,  \notag \\
\sigma _{0} &=&a_{0}+\frac{a_{2}a_{2}^{\prime \prime }+a_{1}a_{2}^{\prime }}{%
8a_{2}}-\frac{a_{1}^{2}+a_{2}a_{1}^{\prime }}{4a_{2}}-\frac{7}{64}\frac{%
a_{2}^{\prime 2}}{a_{2}}.  \notag
\end{eqnarray}

Remark that tensors (\ref{sigmadiffeo} and \ref{sigmadiffeo3}) are
invariants of scalar differential operators with respect to the
diffeomorphism group and tensors (\ref{sigmaauto2}) are invariants of the
group of automorphisms.

In general case we get the following result (cf. \cite{Wil}).

\begin{theorem} $\phantom{\alpha}$
\begin{enumerate}
\item Let $A\in \mathbf{Diff}_{k}\left( \mathbb{R}\right) $ be an ordinary
differential operator of constant type and let $\nabla ^{W}$ be the
associated Wagner connection%
\begin{equation*}
\nabla _{\partial }^{W}\left( \rm{smbl}A\right) =0,
\end{equation*}%
and $Q^{w}:\Sigma _{k}\left( \mathbb{R}\right) \rightarrow \mathbf{Diff}%
_{k}\left( \mathbb{R}\right) $ be the quantization, defined by $\nabla ^{W}.$
Then the total symbol%
\begin{eqnarray*}
\sigma _{\cdot } &=&\sigma _{k}+\sigma _{k-1}+\cdots +\sigma _{0}, \\
\sigma _{i} &\in &\Sigma _{i}\left( \mathbb{R}\right) ,
\end{eqnarray*}%
defined by the condition 
\begin{equation*}
Q^{W}\left( \sigma _{\cdot }\right) =A,
\end{equation*}%
is a tensor invariant of scalar ordinary differential operators with respect
to diffeomorphism group $\mathcal{G}\left( \mathbb{R}\right) $.

Moreover, functions $\sigma _{0}$ and $\lambda _{i},$ $\sigma _{i}=\lambda
_{i}\sigma _{1}^{i}$ are scalar differential invariants$.$

\item Let $A\in \mathbf{Diff}_{k}\left( \xi \right) $ be a linear
differential operator, acting in a line bundle $\xi $ over $\mathbb{R}$ and
having the constant type. Let $\nabla ^{W}$ be the associated Wagner
connection, and let $\nabla ^{A}$ be the linear connection in the line
bundle $\xi $ defined by (\ref{A-ordconnection}) and $Q^{w,A}:\Sigma
_{k}\left( \mathbb{R}\right) \rightarrow \mathbf{Diff}_{k}\left( \xi \right) 
$ be the quantization associated with these connections.

Then the total symbol $\sigma _{\cdot },\ Q^{W,A}\left( \sigma _{\cdot
}\right) =A,$ is an invariant tensor of the ordinary differential operators
with respect to automorphism group $\mathbf{Aut}(\xi )$ and the defined
above functions $\sigma _{0}$ and $\lambda _{i},i=2,...k$ are scalar
differential invariants.
\end{enumerate}
\end{theorem}

\begin{remark}
It is worth to note that these differential invariants are not Wilczynski
invariants \cite{Wil}: they are invariants of operators but not invariants
of equations.
\end{remark}

\subsection{Differential invariants of constant type symbols}

Let $\pi :S^{k}T\left( M\right) \rightarrow M$ be the bundle of symmetric $k$%
-vectors (symbols) and let $\nu _{k}\in \Sigma _{k}\left( \pi \right) $ be
the universal symbol (of order 0). We denote by $\mathcal{O}_{0}\subset
J^{0}\left( \pi \right) $ the domain of regular symbols. The symbols having
the constant type $\mathcal{\varpi }$ constitute a subbundle 
\begin{equation*}
\pi ^{\mathcal{\varpi }}:\mathcal{E}_{\mathcal{\varpi }}\rightarrow M
\end{equation*}%
of the bundle $\pi :\mathcal{O}_{0}\rightarrow M$ of regular symbols.

Then the Wagner connection defines a total covariant differential 
\begin{equation*}
\widehat{d}_{\mathcal{\varpi }}:\Sigma _{1}\left( \pi ^{\mathcal{\varpi }%
}\right) \rightarrow \Sigma _{1}\left( \pi ^{\mathcal{\varpi }}\right)
\otimes \Omega ^{1}\left( \pi ^{\mathcal{\varpi }}\right) ,
\end{equation*}%
over the domain of regular symbols, and, by the construction 
\begin{equation*}
\widehat{d}_{\mathcal{\varpi }}\left( \nu _{k}\right) =0.
\end{equation*}

Let $T^{\mathcal{\varpi }}\in \Omega ^{2}\left( \pi ^{\mathcal{\varpi }%
}\right) \otimes \Sigma _{1}\left( \pi ^{\mathcal{\varpi }}\right) $ be the
total torsion of the connection and $\theta ^{\mathcal{\varpi }}\in \Omega
^{1}\left( \pi ^{\mathcal{\varpi }}\right) $ be the torsion form.

Then, applying the total differential of the dual (to Wagner) connection 
\begin{equation*}
\widehat{d}_{\mathcal{\varpi }}^{\ast }:\Omega ^{1}\left( \pi ^{\mathcal{%
\varpi }}\right) \rightarrow \Omega ^{1}\left( \pi ^{\mathcal{\varpi }%
}\right) \otimes \Omega ^{1}\left( \pi ^{\mathcal{\varpi }}\right)
\end{equation*}%
we get tensor 
\begin{equation*}
\widehat{d}_{\mathcal{\varpi }}^{\ast }\left( \theta ^{\mathcal{\varpi }%
}\right) \in \Omega ^{1}\left( \pi ^{\mathcal{\varpi }}\right) \otimes
\Omega ^{1}\left( \pi ^{\mathcal{\varpi }}\right) .
\end{equation*}

Taking the symmetric $g^{\mathcal{\varpi }}$ and antisymmetric $a^{\mathcal{%
\varpi }}$ parts of this we get tensors%
\begin{equation*}
g^{\mathcal{\varpi }}\in \Sigma ^{2}\left( \pi ^{\mathcal{\varpi }}\right)
,\ \ a^{\mathcal{\varpi }}\in \Omega ^{2}\left( \pi ^{\mathcal{\varpi }%
}\right) .
\end{equation*}%
Assuming that tensor $g^{\mathcal{\varpi }}$ is non degenerated we get total
operator 
\begin{equation*}
A^{\mathcal{\varpi }}\in \Sigma _{1}\left( \pi ^{\mathcal{\varpi }}\right)
\otimes \Omega ^{1}\left( \pi ^{\mathcal{\varpi }}\right) ,
\end{equation*}%
instead of $a^{\mathcal{\varpi }},$ and horizontal $1$-forms%
\begin{equation}
\theta _{1}^{\mathcal{\varpi }}=\theta ^{\mathcal{\varpi }},\theta
_{2}^{w}=A^{\mathcal{\varpi }}(\theta _{1}^{\mathcal{\varpi }}),...,\theta
_{n}^{\mathcal{\varpi }}=A^{\mathcal{\varpi }}(\theta _{n-1}^{\mathcal{%
\varpi }}).  \label{invCoframe}
\end{equation}%
Remark that the torsion $T^{\mathcal{\varpi }}$ and torsion form has order 1
and therefore, tensors $g^{\mathcal{\varpi }},a^{\mathcal{\varpi }},A^{%
\mathcal{\varpi }}$ and $\theta _{i}^{\mathcal{\varpi }}$ has order 2.

We say that a domain $\mathcal{O}_{2}^{\mathcal{\varpi }}\subset J^{2}\left(
\pi ^{\mathcal{\varpi }}\right) $ consist of regular 2-jet of symbols if the
tensor $g^{\mathcal{\varpi }}$ is non degenerated and 
\begin{equation}
\theta _{1}^{\mathcal{\varpi }}\wedge \cdots \wedge \theta _{n}^{\mathcal{%
\varpi }}\neq 0.  \label{reg2jetSymb}
\end{equation}

Let $\left( e_{1}^{\mathcal{\varpi }},...,e_{n}^{\mathcal{\varpi }}\right) $
be the frame of horizontal vector fields $e_{i}^{\mathcal{\varpi }}\in
\Sigma _{1}\left( \pi ^{\mathcal{\varpi }}\right) $ dual to coframe $\left(
\theta _{1}^{\mathcal{\varpi }},...,\theta _{n}^{\mathcal{\varpi }}\right) .$
Then coefficients $J_{\alpha }^{\mathcal{\varpi }}$ in the decomposition of
universal symbol $\nu _{k}$ in this frame 
\begin{equation}
\nu _{k}=\sum_{\left\vert \alpha \right\vert =k}J_{\alpha }^{\mathcal{\varpi 
}}\left( e_{1}^{\mathcal{\varpi }}\right) ^{\alpha _{1}}\cdot \cdots \cdot
\left( e_{n}^{\mathcal{\varpi }}\right) ^{\alpha _{n}},  \label{invarSymbol}
\end{equation}%
are rational functions over regular domain $\mathcal{O}_{2}^{\mathcal{\varpi 
}}$ and invariants of the diffeomorphism group.

\begin{theorem}
\label{SymbNatInv}The field of rational natural invariants of symbols having
degree $k$ and constant type $\mathcal{\varpi }$ is generated by invariants $%
J_{\alpha }^{\mathcal{\varpi }},$ $\left\vert \alpha \right\vert =k,$ and
invariant derivations $e_{i}^{\mathcal{\varpi }},i=1,..,n.$
\end{theorem}

\subsection{Differential invariants of constant type scalar differential
operators}

Let $\chi _{k}^{\mathcal{\varpi }}$: $\emph{Diff}_{k}^{\mathcal{\varpi }%
}(M)\rightarrow M$ be the bundle of scalar differential operator of order $k$%
, having symbols of constant type $\mathcal{\varpi },$ and let $\mathbf{Diff}%
_{k}^{\mathcal{\varpi }}(M)$ be its module of smooth sections.

By $\widehat{\mathcal{O}}_{2}^{\mathcal{\varpi }}\subset J^{2}\left( \chi
_{k}^{\mathcal{\varpi }}\right) $ we denote the domain, where 2-jets of
symbols are regular in the above sense, i.e. 2-jets of symbols belong to
regular domain $\mathcal{O}_{2}^{\mathcal{\varpi }}.$

Denote by $\tau _{k}^{\mathcal{\varpi }}:S^{k}T^{\mathcal{\varpi }%
}\rightarrow M$ bundles \ of symbols having degree $k$ and constant type $%
\mathcal{\varpi },$ and let $\tau _{l}:S^{l}T\rightarrow M$ be bundles of
symbols of degree $l,$ $l=0,1,..,$ and let 
\begin{equation*}
\tau _{\left( k\right) }=\tau _{k}^{\mathcal{\varpi }}\oplus \tau
_{k-1}\oplus \cdots \oplus \tau _{1}\oplus \tau _{0}
\end{equation*}%
be the bundle of total symbols with principle symbol having of constant type 
$\mathcal{\varpi }.$

Consider differential operator 
\begin{equation*}
\mu _{k}:J^{k+1}\left( \chi _{k}^{\mathcal{\varpi }}\right) \rightarrow \tau
_{\left( k\right) },
\end{equation*}%
which sends differential operators $A\in \mathbf{Diff}_{k}^{\mathcal{\varpi }%
}(M)$ having regular 2-jet $[A]_{p}^{2}\in \widehat{\mathcal{O}}_{2}^{%
\mathcal{\varpi }}$ to the total symbol 
\begin{equation*}
\rm{smbl}_{\left( k\right) }\left( A\right) =\left( \rm{%
smbl}_{k}\left( A\right) ,\rm{smbl}_{k-1}\left(
A\right) ,...,\rm{smbl}_{0}\left( A\right) \right) 
\end{equation*}

with respect to the Wagner connection that corresponds to the regular
principal symbol $\rm{smbl}_{k}\left( A\right) .$

It follows from the construction of the Wagner connection that this operator
has order $\left( k+1\right) $ and is natural, i.e. commutes with the action
of the diffeomorphism group.

Regularity conditions allow us to construct invariant coframe (\ref%
{invCoframe}), and then by decomposing (\ref{invarSymbol}) the total symbol
in this coframe to find natural rational invariants $J_{\alpha }^{\mathcal{%
\varpi }},$ where $\left\vert \alpha \right\vert \leq k,$ on the $\left(
k+1\right) $-jet bundle $J^{k+1}\left( \chi _{k}^{\mathcal{\varpi }}\right)
. $

It follows from (\ref{SymbNatInv}) that invariants $J_{\alpha }^{\mathcal{%
\varpi }}$ and invariant derivations $e_{i}^{\mathcal{\varpi }}$ generate
the field of natural invariants of total symbols.

Therefore, applying the prolongations of $\mu _{k},$ 
\begin{equation*}
\mu _{k}^{\left( l\right) }:J^{k+l+1}\left( \chi _{k}^{\mathcal{\varpi }%
}\right) \rightarrow J^{l}(\tau _{\left( k\right) }),
\end{equation*}%
we'll get natural invariants of differential operators of the constant type.

\begin{theorem}
\label{ThInvariantScalar}The field of natural differential invariants of
linear scalar differential operators of order $k\geq 3,$ having constant
type $\mathcal{\varpi },$ is generated by the basic invariants $\mu
_{k}^{\ast }\left( J_{\alpha }^{\mathcal{\varpi }}\right) ,$ $\left\vert
\alpha \right\vert \leq k,$ and invariant derivatives $e_{i}^{\mathcal{%
\varpi }},$ $i=1,...,n.$
\end{theorem}

\subsection{Differential invariants of constant type differential operators,
acting in line bundles}

In this section we consider $\mathbf{Aut}(\xi )$- invariants of linear
differential operators, acting in line bundle $\xi $.

We will consider differential operators of the constant type $A\in \mathbf{Diff%
}_{k}^{\mathcal{\varpi }}\left( \xi \right) \subset \mathbf{Diff}_{k}\left(
\xi \right) $, i.e. operators such that their principal symbol $\rm{smbl%
}_{k}\left( A\right) \in \Sigma _{k}\left( M\right) $ belongs to type $%
\mathcal{\varpi }.$

We denote by 
\begin{equation*}
\pi _{k}^{\mathcal{\varpi }}:\emph{Diff}_{k}^{\mathcal{\varpi }}(\xi
)\rightarrow M
\end{equation*}%
the bundle of the constant type operators.

We'll also assume that the symbols of operators are regular in the previous
sense:%
\begin{equation*}
j_{2}\left( \rm{smbl}_{k}\left( A\right) \right) \subset 
\widehat{\mathcal{O}}_{2}^{\mathcal{\varpi }}.
\end{equation*}%
Thus the symbol defines the Wagner connection $\nabla ^{w}$, the invariant
coframe and, in addition, the linear connection $\nabla ^{A}$ in the line
bundle $\xi ,$ if the symbol is Wagner regular.

From now on we'll call such symbols simply \textit{regular. }

Remark that Wagner regularity depends on the second jet of the operator and
therefore defines an open subset in the space of second jets of operators $%
\mathcal{O}_{2}^{\mathcal{\varpi }}\subset J^{2}\left( \pi _{k}^{\mathcal{%
\varpi }}\right) $.

By the construction all of these connections are $\mathbf{Aut}(\xi )$%
-invariants.

Therefore, there is the total symbol 
\begin{equation*}
\rm{smbl}_{\left( k\right) }\left( A\right) =\left( \rm{%
smbl}_{k}\left( A\right) ,\rm{smbl}_{k-1}\left(
A\right) ,...,\rm{smbl}_{0}\left( A\right) \right) \in \tau
_{\left( k\right) },
\end{equation*}%
such that 
\begin{equation*}
Q^{A}\left( \rm{smbl}_{\left( k\right) }\left( A\right)
\right) =A.
\end{equation*}%
Here we denoted by $Q^{A}$ the quantization, defined by connections\\ $\nabla
^{\rm{smbl}_{k}\left( A\right) }$ and $\nabla ^{A}.$

Remark that operator $Q^{A}$ is also $\mathbf{Aut}(\xi )$-invariant.

Let $\kappa ^{A}\in \Omega ^{2}\left( M\right) $ be the curvature form of $%
\nabla ^{A}.$ It depends on the second jets of the operator and by the
standard procedure defines a horizontal $2$-form $\kappa \in \Omega
^{2}\left( \pi _{k}^{\mathcal{\varpi }}\right) ,$ having the second order
and satisfying the universality condition:%
\begin{equation*}
j_{2}\left( A\right) ^{\ast }\left( \kappa \right) =\kappa ^{A},
\end{equation*}%
for all differential operators $A\in \mathbf{Diff}_{k}^{\mathcal{\varpi }%
}\left( \xi \right) $ with regular symbols.

As above, we consider differential operator 
\begin{equation*}
\mu _{k}:J^{k+1}\left( \pi _{k}^{\mathcal{\varpi }}\right) \rightarrow \tau
_{\left( k\right) },
\end{equation*}%
which sends differential operators $A\in \mathbf{Diff}_{k}^{\mathcal{\varpi }%
}(M),$ having regular 2-jet $[A]_{p}^{2}\in \mathcal{O}_{2}^{\mathcal{\varpi 
}},$ to their total symbol and get the following result similar to (\ref%
{ThInvariantScalar}).

\begin{theorem}
\label{ThInvariantBundle}The field of natural differential $\mathbf{Aut}(\xi
)$-invariants of linear differential operators of order $k\geq 3,$ having
constant type $\mathcal{\varpi }$ and acting in line bundle $\xi ,$ is
generated by the basic invariants $\mu _{k}^{\ast }\left( J_{\alpha }^{%
\mathcal{\varpi }}\right) ,$ $\left\vert \alpha \right\vert \leq k,K_{ij},$
where $J_{\alpha }^{\mathcal{\varpi }}$ and $K_{ij}$ are coordinates of the
total symbol and the universal curvature for $\kappa $ in the invariant
frame, and \ by invariant derivatives $e_{i}^{\mathcal{\varpi }},$ $%
i=1,...,n.$
\end{theorem}

\section{Equivalence of differential operators}

\subsection{Equivalence of scalar differential operators}

Let $A$ be a linear scalar differential operator over $M$, i.e. $A\in \mathbf{Diff}_k( M)$.
We'll say this operator is in \textit{general position }if for any
point $a\in M$ there are natural invariants $I_{1},...,I_{n},$ where $n=\dim
M,$ such that their values $I_{i}\left( A\right) ,$ $i=1,...,n,$ on this
operator are independent in a neighborhood $U$ of this point, i.e. 
\begin{equation}
dI_{1}\left( A\right) \wedge \cdots \wedge dI_{n}\left( A\right) \neq 0.
\label{InvIndepen}
\end{equation}%
The principle of $n$-invariants states that these invariants and invariants $%
I_{\alpha }=\square \left( I_{1}^{\alpha _{1}}\cdots I_{n}^{\alpha
_{n}}\right) ,$ $\left\vert \alpha \right\vert \leq k,$ generate all
invariants and dependencies%
\begin{equation}
I_{\alpha }\left( A\right) =F_{\alpha }\left( I_{1}\left( A\right) ,\cdots
,I_{n}\left( A\right) \right)   \label{diffdepend}
\end{equation}%
give us coefficients of the operator in these local coordinates 
\begin{equation}
x_{1}=I_{1}\left( A\right) ,....,x_{n}=I_{n}\left( A\right) .
\label{natural coord}
\end{equation}

We call these coordinates \textit{natural }due to the following\textit{\ }%
their property.

Let \textit{\ }$A^{\prime }\in \mathbf{Diff}_{k}\left( M^{\prime }\right) $
be another operator and let $\phi :M\rightarrow M^{\prime }$ be a local
diffeomorphism such that 
\begin{equation*}
\phi _{\ast }\left( A\right) =A^{\prime },\phi \left( a\right) =a^{\prime }.
\end{equation*}%
Then we have 
\begin{equation*}
\phi _{\ast }\left( I_{i}\left( A\right) \right) =\phi ^{\ast -1}\left(
I_{i}\left( A\right) \right) =I_{i}\left( \phi _{\ast }\left( A\right)
\right) =I_{i}\left( A^{\prime }\right) ,
\end{equation*}%
i.e. in natural coordinates the diffeomorphism has the following form%
\begin{equation*}
\left( x_{1},...,x_{n}\right) \rightarrow \left( x_{1}^{\prime
},...,x_{n}^{\prime }\right) ,
\end{equation*}%
where 
\begin{equation*}
x_{1}^{\prime }=I_{1}\left( A^{\prime }\right) ,....,x_{n}^{\prime
}=I_{n}\left( A^{\prime }\right) .
\end{equation*}%
Therefore functions 
\begin{equation*}
F_{\alpha }=F_{\alpha }\left( x_{1},...,x_{n}\right)
\end{equation*}%
defines the orbit of the germ of operator $A$ at the point $a\in M$ with
respect to the diffeomorphism group.

We will reformulate this observation in the following way.

For a given operator $A\in \mathbf{Diff}_{k}\left( M\right) $ in general
position we"ll consider:

\begin{itemize}
\item \textit{Natural charts - i.e. }local diffeomorphisms 
\begin{equation*}
\phi ^{I}:U^{I}\rightarrow \mathbf{D}^{I}\subset \mathbb{R}^{n},
\end{equation*}%
on open domains in $\mathbb{R}^{n}$ given by such natural invariants $%
I=\left( I_{1},...,I_{n}\right) ,$ that (\ref{InvIndepen}) holds in open set 
$U^{I}.$

\item \textit{Natural atlas }- i.e. a collection of natural charts $\left\{
U^{I},\phi ^{I}\right\} ,$ covering manifold $M,$ and given by distinct
natural invariants.
\end{itemize}

We denote by 
\begin{equation*}
\mathbf{D}^{IJ}=\phi ^{I}\left( U^{I}\cap U^{J}\right)
\end{equation*}%
and assume that domains $U^{I},U^{IJ}$ are connected and simply connected.

Let $A_{I}=\phi _{\ast }^{I}\left( \left. A\right\vert _{U^{I}}\right) ,$ $%
A_{IJ}=\phi _{\ast }^{I}\left( \left. A\right\vert _{U^{I}\cap U^{J}}\right) 
$ be the images of the operator $A$ in natural coordinates.

Then $\phi _{\ast }^{IJ}\left( A_{IJ}\right) =A^{JI},$ where $\phi ^{IJ}:%
\mathbf{D}^{IJ}\rightarrow \mathbf{D}^{JI}$ are the transition maps.

We call \cite{LY3} such atlas as natural atlas associated with operator $A$
and collection the $\left( \mathbf{D}^{I},\mathbf{D}^{IJ},\phi
^{IJ},A^{I},A^{IJ}\right) =\mathcal{D}\left( A\right) $ we call -\textit{%
natural model} of the operator.

\begin{theorem}
Let $A,A^{\prime }\in \mathbf{Diff}_{k}\left( M\right) $ be operators in
general position. Then these operators are equivalent with respect to group
of diffeomorphisms if and only if the following conditions hold:\newline
Open sets 
\begin{equation*}
U^{I\prime }=\left( \phi ^{I\prime }\right) ^{-1}\left( \mathbf{D}%
^{I}\right) ,
\end{equation*}%
\newline
where $\phi ^{I\prime }=\left( I_{1}\left( A^{\prime }\right)
,...,I_{n}\left( A^{\prime }\right) \right) :M\rightarrow \mathbb{R}^{n}$
constitute a natural atlas for operator $A^{\prime },$ $\phi ^{IJ\prime
}=\phi ^{IJ}:\mathbf{D}^{IJ}\rightarrow \mathbf{D}^{JI},$ and
\begin{equation*}
A_{I}=\phi _{\ast }^{I\prime }\left( \left. A^{\prime }\right\vert
_{U^{I\prime }}\right) ,A_{IJ}=\phi _{\ast }^{I\prime }\left( \left.
A^{\prime }\right\vert _{U^{I\prime }\cap U^{J\prime }}\right) ,
\end{equation*}%
i.e. when natural models of the operators coincide.
\end{theorem}

\begin{proof}
As we have seen any diffeomorphism, transforming operator $A$ to $A^{\prime
},$ in natural coordinates has the form of the identity map.
\end{proof}

\begin{remark}
We have two types of scalar differential operators, where we able to compute
the fields of natural differential invariants: operators with symbols of the
general type and the constant type operators.
\end{remark}

\subsection{Equivalence of differential operators, acting in line bundles}

For the case of operators $A\in \mathbf{Diff}_{k}\left( \xi \right) ,$
acting in the line bundle $\xi ,$ we"ll restrict ourselves by the regular
operators of the constant type only.

In this case we have two connections naturally associated with operator:
Wagner connection $\nabla ^{\rm{smbl}_{k}\left( A\right) }$ on
the manifold and connection $\nabla ^{A}$ in the line bundle. These
connections define also the quantization $Q^{A},$ naturally (with respect to
the automorphism group $\mathbf{Aut}(\xi )$) associated with this operator.

Thus operator $A$ has well defined total symbol $\rm{smbl}%
_{\left( k\right) }\left( A\right) $ and accordingly \textit{scalar
shadow} $A_{\natural }\in \mathbf{Diff}_{k}\left( M\right) $ of the operator 
$A,$%
\begin{equation*}
A_{\natural }=Q^{W}\left( \rm{smbl}_{\left( k\right) }\left(
A\right) \right) .
\end{equation*}%
Here $Q^{W}$ is the quantization for scalar differential operators given by
the Wagner connection $\nabla ^{\rm{smbl}_{k}\left( A\right) }.
$

Naturality of all these constructions show us that if two regular operators $%
A,B\in \mathbf{Diff}_{k}\left( \xi \right) $ of constant type $\mathcal{%
\varpi }$ are $\mathbf{Aut}(\xi )$-equivalent then their scalar shadows $%
A_{\natural },B_{\natural }\in \mathbf{Diff}_{k}\left( M\right) $ should be
equivalent with respect to the diffeomorphism group $\mathcal{G}\left(
M\right) .$

On the other hand, let operators $A_{\natural },B_{\natural }\in \mathbf{Diff%
}_{k}\left( M\right) $ are $\mathcal{G}\left( M\right) $-equivalent and let
diffeomorphism $\psi :M\rightarrow M$ (see the above  construction ) sends
operator $A_{\natural }\ $to $B_{\natural },$ $\psi _{\ast }\left(
A_{\natural }\right) =B_{\natural }.$

Then diffeomorphism $\psi $ has a lift $\overline{\psi }\in \mathbf{Aut}(\xi
)$ if and only if (see, Proposition \ref{diffeomorphismLift}) 
\begin{equation}
\psi ^{\ast }\left( w_{1}\left( \xi \right) \right) =w_{1}\left( \xi \right)
.  \label{StifelCond}
\end{equation}

Assume that this condition holds, then operators $\overline{A}=\overline{%
\psi }_{\ast }\left( A\right) $ and $B$ has the same total symbols and
scalar shadows.

Let $\overline{\nabla }^{A}=\nabla ^{\overline{A}}$ be the image of
connection $\nabla ^{A}$ under the automorphism $\overline{\psi },$ and let $%
\kappa _{A}\in \Omega ^{2}\left( M\right) $ and $\kappa _{B}\in \Omega
^{2}\left( M\right) $ be curvature forms for linear connections $\nabla ^{A}$
\ \ and $\nabla ^{B}$ respectively . They give us one more condition for
diffeomorphism $\psi :$%
\begin{equation}
\psi ^{\ast }\left( \kappa _{B}\right) =\kappa _{A},
\label{curvature condition}
\end{equation}%
or equivalently 
\begin{equation*}
\kappa _{B}=\kappa _{\overline{A}}.
\end{equation*}

Then,%
\begin{equation*}
d_{\nabla ^{\overline{A}}}-d_{\nabla ^{B}}=\theta _{\psi }\otimes \rm{id%
}
\end{equation*}%
for some differential 1-form $\theta _{\psi }\in \Omega ^{1}\left( M\right)
, $ and therefore 
\begin{equation*}
d\theta _{\psi }=0,
\end{equation*}%
if condition (\ref{curvature condition}) holds.

Remark, that the different lifts $\overline{\psi }$ differ on automorphisms
given by multiplication on functions $f\in \mathcal{F}\left( M\right) $ that
induce transformations $\theta _{\psi }\rightarrow \theta _{\psi }+d\ln
\left\vert f\right\vert .$

Thus the cohomology class 
\begin{equation*}
\vartheta _{A,B}\in H^{1}\left( M,\mathbb{R}\right)
\end{equation*}%
of the closed 1-form $\theta _{\psi }$ defines a new obstruction for
existence of the automorphism $\psi .$

Summarizing, we get the following.

\begin{theorem}
Two regular operators $A,B\in \mathbf{Diff}_{k}\left( \xi \right) $ of
constant type $\mathcal{\varpi }$ are $\mathbf{Aut}(\xi )$-equivalent if and
only if their scalar shadows $A_{\natural },B_{\natural }\in \mathbf{Diff}%
_{k}\left( M\right) $ are $\mathcal{G}\left( M\right) $-equivalent and the
diffeomorphism $\psi :M\rightarrow M,$ $\psi _{\ast }\left( A_{\natural
}\right) =B_{\natural },$ satisfies in addition to the following conditions:

\begin{enumerate}
\item It preserves the first Stiefel-Whitney class $w_{1}\left( \xi \right) $
of the bundle:%
\begin{equation*}
\psi ^{\ast }\left( w_{1}\left( \xi \right) \right) =w_{1}\left( \xi \right)
.
\end{equation*}

\item It transforms the curvature form of the connection $\nabla ^{B}$ to
the connection form of the connection $\nabla ^{A}:$%
\begin{equation*}
\psi ^{\ast }\left( \kappa _{B}\right) =\kappa _{A}.
\end{equation*}

\item The obstruction $\vartheta _{A,B}\in H^{1}\left( M,\mathbb{R}\right) $
is trivial:%
\begin{equation*}
\vartheta _{A,B}=0.
\end{equation*}
\end{enumerate}
\end{theorem}

\begin{remark}
In the above construction of the diffeomorphism $\psi $ in terms of natural
atlases we should add the following condition only: components of the
curvature forms in the corresponding natural charts should be equal. This
condition shall fix the local structure of the diffeomorphism. Topological
conditions: $\psi ^{\ast }\left( w_{1}\left( \xi \right) \right)
=w_{1}\left( \xi \right) $ and $\vartheta _{A,B}=0$ should be checked
additionally.
\end{remark}

\subsection{Equivalence of differential equations}

By differential equations in this section we'll understand homogeneous
differential equations given by Chern regular differential operators $A\in 
\mathbf{Diff}_{k}\left( \xi \right) ,$ or in another words, by conformal
classes $[A]$ of operators.

We'll use the Chern connection $\nabla ^{\lbrack \rm{smbl}_{k}\left(
A\right) ]},$ defining on manifold $M$ by the conformal class $[\rm{smbl%
}_{k}\left( A\right) ]$ of the principal symbol, and the linear connection $%
\nabla ^{\lbrack A]},$ defining by the conformal class $[A]$ of the operator.

We denote by 
\begin{equation*}
\rm{smbl}_{\left( k\right) }\left( A\right) =\sum_{i=0}^{k}%
\rm{smbl}_{i}\left( A\right) 
\end{equation*}%
the total symbol of the operator, given by the quantization, associated with
connections  $\nabla ^{\lbrack \rm{smbl}_{k}\left( A\right) ]}$ and $%
\nabla ^{\lbrack A]}.$

The linear property of the quantization shows us that 
\begin{equation*}
\rm{smbl}_{i}\left( fA\right) =f\rm{smbl}%
_{i}\left( A\right) ,
\end{equation*}%
for all functions $f\in \mathcal{F}\left( M\right) .$

In other words, tensors $\rm{smbl}_{i}\left( A\right) \in
\sum_{i}\left( M\right) $ are relative invariants of operators.

Let $\omega ^{C}=d\theta ^{\rm{smbl}_{k}\left( A\right) }\in
\Omega ^{2}\left( M\right) $ and $\widehat{\omega }:\Omega _{1}\left(
M\right) \rightarrow \Omega ^{1}\left( M\right) $ be the operator defined in
(\ref{ChernOperator}).

Then functions 
\begin{equation*}
H_{i}\left( A\right) =\alpha ^{i}\rfloor \rm{smbl}_{i}\left(
A\right) ,
\end{equation*}%
where $\alpha =\widehat{\omega }\left( \rm{smbl}_{1}\left(
A\right) \right) \in \Omega ^{1}\left( M\right) ,$ are also relative
invariants%
\begin{equation*}
H_{i}\left( fA\right) =f^{i+1}H_{i}\left( A\right) .
\end{equation*}%
Then operators 
\begin{equation*}
A_{\nu ,i}=\lambda _{i}\left( A\right) ~A\in \mathbf{Diff}_{k}\left( \xi
\right) ,
\end{equation*}%
where 
\begin{equation*}
\lambda _{i}\left( A\right) =\frac{H_{i}\left( A\right) }{H_{i+1}\left(
A\right) },
\end{equation*}%
we call \textit{normalizations of operator} $A.$

It is easy to check that 
\begin{equation*}
\left( fA\right) _{\nu ,i}=A_{\nu ,i},
\end{equation*}%
and therefore $A_{\nu ,i}$ depends on the conformal class only.

\begin{theorem}
Let $[A]$ and $[B]$ be conformal classes of Chern regular differential
operators $A,B\in \mathbf{Diff}_{k}\left( \xi \right) $, such that $\lambda
_{i}\left( A\right) $, $\lambda _{i}\left( B\right) \in \mathcal{F}\left(
M\right) $, for some $i,$ $i=0,...,k-1.$

Then these classes are $\mathbf{Aut}(\xi )$-equivalent if and only if their
normalizations $A_{\nu ,i}$ and $B_{\nu ,i}$ are $\mathbf{Aut}(\xi )$%
-equivalent.
\end{theorem}
\bigskip

{\bf Acknowledgements}\medskip

This work is supported by the Russian Foundation for Basic Research under grant 18-29-10013 mk.


\bigskip

\begin{thebibliography}{99}
\bibitem{Ak} Akivis Maks A., Three-webs of multidimensional surfaces.
(Russian) Trudy Geometr. Sem. 2, 1969, pp. 7--31.

\bibitem{ALV} Alekseevskij D., Lychagin V. , Vinogradov A., Basic Ideas and
Concepts of Differential Geometry, in: Encyclopedia of Mathematical
Sciences, Geometry 1, vol. 28, Springer, Berlin, 1991.

\bibitem{BL} Bibikov Pavel, Lychagin Valentin, Classification of the linear
actions of algebraic groups on spaces of homogeneous forms, Dokl. Math. 85
(2012), no. 1, 109--112.

\bibitem{BL2} Bibikov Pavel, Lychagin Valentin, Projective classification of
binary and ternary forms. J. Geom. Phys. 61 (2011), no. 10, 1914--1927.

\bibitem{Brog} Louis De Broglie, An introduction to the study of wave
mechanics, London, Methuen \& Co, 253 pp, 1930.

\bibitem{Ch} Chern, S. S., Eine Invariantentheorie der 3-Gewebe aus
r-dimensionalen Mannigfaltigkeiten in R2r Abh. Math. Sem. Univ. Hamburg 11
(1936), 333--358.

\bibitem{Gold} Goldschmidt Hubert, Existence theorems for analytic linear
partial differential equations. Ann. of Math. (2) 86 1967 246--270.

\bibitem{Ibr} Ibragimov N.Kh., Invariants of hyperbolic equations: solutions
of the Laplace problem, J. Appl. Mech. Tech. Phys. 45 (2) (2004) 158--166.

\bibitem{KamOlv} Kamran Niky, Olver Peter, Equivalence of differential
operators, SIAM J. Math. Anal. 20 (5) (1989) 1172--1185.

\bibitem{KN} Kobayashi Shoshichi, Nomidzu Katsumi, Foundations of
differential geometry, vol.1, John Wiley \& Sons, Inc., (1963).

\bibitem{KLV} Krasilshchik I.S., Lychagin V.V., Vinogradov A.M., Geometry of
jet spaces and nonlinear differential equations, Gordon and Bridge, (1986),
441pp.

\bibitem{KrLy} Kruglikov Boris, Lychagin Valentin, Geometry of differential
equations, in: D. Krupka, D. Saunders (Eds.), Handbook of Global Analysis,
Elsevier,Amsterdam, 2008, pp. 725--772.

\bibitem{KL2} Kruglikov Boris, Lychagin Valentin, Global Lie-Tresse theorem,
Selecta Math. (NS) 22 (3) (2016) 1357--1411.

\bibitem{Lap} Laplace P.S., Recherches sur le calcul int\'{e}gral aux diff%
\'{e}rences partielles, in: M\'{e}moires de l'Acad \'{e}mie Royale des
Sciences de Paris (1773/77), pp. 341--402.

\bibitem{LY2} Lychagin Valentin, Yumaguzhin Valeriy, Classification of the
second order linear differential operators and differential equations,
Journal of Geometry and Physics, 130, (2018), pp. 213--228

\bibitem{LY3} Lychagin Valentin, Yumaguzhin Valeriy, On equivalence of third
order linear differential operators on two-dimensional manifolds,
arXiv:1904.08664, submitted to Journal of geometry and physics

\bibitem{LQ} Lychagin Valentin, Quantum mechanics on manifolds, Acta Appl.
Math. 56 (2--3), (1999) pp. 231--251.

\bibitem{MS} Milnor John, Stasheff James, Characteristic classes, Princetion
University Press, 1974, 326 pp.

\bibitem{Ovs} Ovsyannikov L.V., Group properties of the Chaplygin equation,
J. Appl. Mech. Tech. Phys. 3 (1960) 126--145.

\bibitem{Pal} Palais, Richard S. Seminar on the Atiyah-Singer index theorem.
Annals of Mathematics Studies, No. 57, Princeton University Press,
Princeton, N.J. 1965 x+366 pp.

\bibitem{Qui} Quillen Daniel, Formal theory of linear overdetermined systems
of partial differential equations. Thesis (Ph.D.)--Harvard University. 1964

\bibitem{Riem} Riemann Bernard, Gesammelte mathematische werke und
Avissenschaftlicher Nachlass, Vol. XXII, Teubner, Leipzig, 1876, pp.
357--370.

\bibitem{Ros} Rosenlicht M., A remark on quotient spaces, An. Acad. Brasil.
Ci\^{e}nc. 35 (1963) 487--489.

\bibitem{Sp} Spencer, D. C. Overdetermined systems of linear partial
differential equations. Bull. Amer. Math. Soc. 75 1969 179--239.

\bibitem{Wag} Wagner V. V., Two dimensional space with cubic metric, Sci.
notes of Saratov State University, Vol. 1(XIV), Ser. FMI, No.1, 1938. (in
Russian).

\bibitem{Wil} Wilczynski E.J., Projective Differential Geometry of Curves
and Ruled Surfaces, Teubner, Leipzig, 1905.
\end{thebibliography}
\end{document}